\documentclass[11pt]{amsart}
\usepackage{import}
\usepackage[margin=1in]{geometry}
\usepackage{algorithm}
\usepackage[noend]{algorithmic}
\usepackage{amsmath,amsfonts,amsthm,amssymb}
\usepackage[english]{babel}
\usepackage{hyperref,nameref,cleveref}
\usepackage{enumitem}
\usepackage{gensymb}
\usepackage{graphicx}
\usepackage{latexsym}
\usepackage{listings}
\usepackage{mathtools}
\usepackage{multicol,multirow}
\usepackage[numbers]{natbib}
\usepackage{nicefrac}
\usepackage{setspace} 
\usepackage{subcaption}
\usepackage{times}
\usepackage{thmtools}
\usepackage{upquote}
\usepackage{url}
\usepackage{wrapfig}
\usepackage{xcolor}

\setlist{nosep, leftmargin=*}

\usepackage[usestackEOL]{stackengine}

\usepackage{tikz} 
\usetikzlibrary{shapes.geometric}
\usetikzlibrary{arrows.meta,arrows}
\usetikzlibrary{trees} 

\usepackage[framed,numbered,autolinebreaks,useliterate]{mcode}
\usepackage{algorithm}
\usepackage[noend]{algorithmic}

\floatname{algorithm}{Algorithm}

\input{format}

\begin{document}

\strutlongstacks{T}
\title[Simpler is better: a comparative study of randomized  algorithms for CUR and ID]{Simpler is better: a comparative study of randomized pivoting algorithms
for CUR and interpolative decompositions}
\author{Yijun Dong \and Per-Gunnar Martinsson}
\address{Oden Institute for Computational Engineering \& Sciences, University of Texas at Austin}

\maketitle

\begin{center}
\begin{minipage}{120mm}
\textbf{Abstract:}
Matrix skeletonizations like the interpolative and CUR decompositions provide a framework for low-rank approximation in which subsets of a given matrix's columns and/or rows are selected to form approximate spanning sets for its column and/or row space.
Such decompositions that rely on ``natural'' bases have several advantages over traditional low-rank decompositions with orthonormal bases, including preserving properties like sparsity or non-negativity, maintaining semantic information in data, and reducing storage requirements.
Matrix skeletonizations can be computed using classical deterministic algorithms such as column pivoted QR, which work well for small-scale problems in practice, but suffer from slow execution as the dimension increases and can be vulnerable to adversarial inputs.
More recently, randomized pivoting schemes have attracted much attention, as they have proven capable of accelerating practical speed, scale well with dimensionality, and sometimes also lead to better theoretical guarantees. 
This manuscript provides a comparative study of various randomized pivoting based matrix skeletonization algorithms that leverage classical pivoting schemes as building blocks.
We propose a general framework that encapsulates the common structure of these randomized pivoting based algorithms, and provide an a-posteriori-estimable error bound for the framework. Additionally, we propose a novel concretization of the general framework, and numerically demonstrate its superior empirical efficiency.

\vspace{2mm}

\end{minipage}
\end{center}

\section{Introduction}

The problem of computing a low-rank approximation to a matrix is a classical one that has drawn increasing attention due to its importance in the analysis of large data sets.
At the core of low-rank matrix approximation is the task of constructing bases that approximately span the column and/or row spaces of a given matrix.
This manuscript investigates algorithms for low-rank matrix approximations with ``natural bases'' of the column and row spaces -- bases formed by selecting subsets of the actual columns and rows of the matrix.
To be precise, given an $m \times n$ matrix $\Ab$ and a target rank $k < \min(m,n)$, we seek to determine an $m \times k$ matrix $\Cb$ holding $k$ of the columns of $\Ab$, and a $k \times n$ matrix $\Zb$ such that
\begin{equation}
\label{eq:colID}
\begin{array}{ccccccc}
\Ab &\approx& \Cb & \Zb.\\
m\times n && m\times k & k\times n
\end{array}
\end{equation}
We let $J_{\rm s}$ denote the index vector of length $k$ that
identifies the $k$ chosen columns, so that, in MATLAB notation,
\begin{equation}
\label{eq:Js}
\Cb = \Ab(\colon,J_{\rm s}).
\end{equation}
If we additionally identify an index vector $I_{\rm s}$ that marks a subset of the rows that forms an approximate basis for the row space of $\Ab$, we can then form the ``CUR'' decomposition
\begin{equation}
\label{eq:CUR}
\begin{array}{ccccccc}
\Ab & \approx & \Cb & \Ub & \Rb,\\
m\times n && m\times k & k\times k & k\times n
\end{array}
\end{equation}
where $\Ub$ is a $k\times k$ matrix, and
\begin{equation}
\label{eq:Is}
\Rb = \Ab(I_{\rm s},\colon).
\end{equation}
The decomposition (\ref{eq:CUR}) is also known as a ``matrix skeleton'' \cite{goreinov1997} approximation (hence  the subscript ``s'' for ``skeleton'' in $I_{\rm s}$ and $J_{\rm s}$).
Matrix decompositions of the form (\ref{eq:colID}) or (\ref{eq:CUR}) possess several compelling properties:
(i) Identifying $I_{\rm s}$ and/or $J_{\rm s}$ is often helpful in data interpretation.
(ii) The decompositions (\ref{eq:colID}) and (\ref{eq:CUR}) preserve important properties of the matrix $\Ab$. For instance, if $\Ab$ is sparse/non-negative, then $\Cb$ and $\Rb$ are also sparse/non-negative.
(iii) The decompositions (\ref{eq:colID}) and (\ref{eq:CUR}) are often memory efficient.
In particular, when the entries of $\Ab$ itself are available, or can inexpensively be computed or retrieved, then once $J_{\rm s}$ and $I_{\rm s}$ have been determined, there is no need to store $\Cb$ and $\Rb$ explicitly.

Deterministic techniques for identifying close-to-optimal index vectors $I_{\rm s}$ and $J_{\rm s}$ are well established. Greedy algorithms
such as the classical column pivoted QR (\pqr) \cite[Sec.~5.4.1]{golub2013},
and variations of LU with complete pivoting \cite{trefethen1997, zhao2005} 
often work well in practice. There also exist specialized pivoting schemes that come with strong theoretical performance guarantees \cite{gu1996}.

While effective for smaller dense matrices, classical techniques based on
pivoting become computationally inefficient as the matrix sizes grow. The
difficulty is that a global update to the matrix is in general required
before the next pivot element can be selected. The situation becomes
particularly dire for sparse matrices, as each update tends to create
substantial fill-in of zero entries.

To better handle large matrices, and huge sparse matrices in particular,
a number of algorithms based on
\textit{randomized sketching} have been proposed in recent years.
The idea is to extract a ``sketch'' $\Yb$ of the matrix that is
far smaller than the original matrix, yet contains enough information
that the index vectors $I_{\rm s}$ and/or $J_{\rm s}$ can be determined
by using the information in $\Yb$ alone. Examples include:

\begin{enumerate}
\item \textit{Discrete empirical interpolation method (DEIM):}
The sketching step consists of computing approximations to the dominant left and right
singular vectors of $\Ab$, for instance using the
randomized SVD (RSVD) \cite{halko2011,Liberty20167}.
Then a greedy pivot based scheme is used to pick the index sets $I_{\rm s}$ and $J_{\rm s}$
\cite{drmac2016qdeim,sorensen2014}.
\item \textit{Leverage score sampling:}
Again, the procedures start by computing approximations to the dominant left and right singular vectors of $\Ab$ through a randomized scheme. Then these approximations are used to compute probability distributions on the row and/or column indices, from which a random subset of columns and/or rows is sampled.
\item \textit{Pivoting on a random sketch:}
With a random matrix $\Omegab \in \mathbb{R}^{k\times m}$ drawn from some appropriate distribution, a sketch of $\Ab$ is formed via $\Yb = \Omegab\Ab$. Then, a classical pivoting strategy such as the \pqr is applied on $\Yb$ to identify a spanning set of columns.
\end{enumerate}

The existing literature \cite{anderson2017,chen2020,cohen2014,derezinski2020improved,derezinski2020dpp,drineas2012,drmac2016qdeim,gu1996,mahoney2009,sorensen2014,voronin2017} presents compelling evidence in support of each of these frameworks, in the form of mathematical theory and/or empirical numerical experiments.

The objective of the present manuscript is to organize different strategies, and to conduct a systematic comparison, with the focus on their empirical accuracy and efficiency.
In particular, we compare different strategies for extracting a random sketch, such as techniques based on
Gaussian random matrices \cite{halko2011,indyk1998,martinsson2020,woodruff2015},
random fast transforms \cite{boutsidis2013srtt,halko2011,martinsson2020, rokhlin2008,tropp2011,woolfe2008}, and
random sparse embeddings \cite{clarkson2017,martinsson2020,meng2013, nelson2013,tropp2017a,woodruff2015}. We also compare different pivoting strategies
such as pivoted QR \cite{golub2013,voronin2017}
versus pivoted LU \cite{chen2020,sorensen2014}. Finally, we compare how well sampling
based schemes perform in relation to pivoting based schemes.

In addition to providing a comparison of existing methods, the manuscript proposes a general framework that encapsulates the common structure shared by some popular randomized pivoting based algorithms, and presents an a-posteriori-estimable error bound for the framework. Moreover, the manuscript introduces a novel concretization of the general framework that is faster in execution than the schemes of \cite{sorensen2014, voronin2017}, while picking equally close-to-optimal skeletons in practice.
In its most basic version, our simplified
method for finding a subset of $k$ columns of $\Ab$ works as follows:

\vspace{2mm}

\begin{center}
\begin{minipage}{0.9\textwidth}
\textit{Sketching step:} Draw $\Omegab \in \mathbb{R}^{k\times m}$
from a Gaussian distribution and form $\Yb = \Omegab\Ab$.

\vspace{2mm}

\textit{Pivoting step:} Perform a \textit{partially pivoted} LU decomposition of
$\Yb^{*}$. Collect the chosen pivot indices in the index vector $J_{\rm s}$.
\end{minipage}
\end{center}

\vspace{2mm}

What is particularly interesting about this process is that while the
LU factorization with partial pivoting (\plu) is \textit{not}
rank revealing for a general matrix $\Ab$, the randomized mixing done
in the sketching step makes \plu excel at picking spanning columns. Furthermore, the randomness introduced by sketching empirically serves as a remedy for the vulnerability of classical pivoting schemes like \plu to adversarial inputs (\eg, the Kahan matrix \cite{kahan1966}).
The scheme can be accelerated further by incorporating a structured
random embedding $\Omegab$. Alternatively, its accuracy can be
enhanced by incorporating one of two steps of power iteration when
building the sample matrix $\Yb$.

The manuscript is organized as following:
\Cref{sec:background} provides a brief overview for the interpolative and CUR decompositions (\Cref{subsec:cur}), along with some essential building blocks of the randomized pivoting algorithms, including randomized linear embeddings (\Cref{subsec:embedding}), randomized low-rank SVD (\Cref{subsec:rsvd}), and matrix decompositions with pivoting (\Cref{subsec:pivoting_mat_decomp}).
\Cref{sec:exist_cur_algo} reviews existing algorithms for matrix skeletonizations (\Cref{subsec:sampling_cur}, \Cref{subsec:pivoting_cur}), and introduces a general framework that encapsulates the structures of some randomized pivoting based algorithms.
In \Cref{sec:cur_rand_lupp}, we propose a novel concretization of the general framework, and provide an a-posteriori-estimable bound for the associated low-rank approximation error.
With the numerical results in \Cref{sec:experiments}, we first compare the efficiency of various choices for the two building blocks in the general framework: randomized linear embeddings (\Cref{subsec:exp_embedding}) and matrix decompositions with pivoting (\Cref{subsec:exp_pivot}). Then, we demonstrate empirical advantages of the proposed algorithm by investigating the accuracy and efficiency of assorted randomized skeleton selection algorithms for the CUR decomposition (\Cref{subsec:exp_cur}).

\section{Background}\label{sec:background}
We first introduce some closely related low-rank matrix decompositions that rely on ``natural'' bases, including the CUR decomposition, and the column, row, and two-sided interpolative decompositions (ID) in \Cref{subsec:cur}.
\Cref{subsec:embedding} describes techniques for computing randomized sketches of matrices, based on which \Cref{subsec:rsvd} discusses the randomized construction of low-rank SVD. \Cref{subsec:pivoting_mat_decomp} describes how these can be used to construct matrix decompositions.
While introducing the background, we include proofs of some well-established facts that provide key ideas, but are hard to extract from the context of relevant references.

\subsection{Notation}\label{subsec:notation}
Let $\Ab \in \R^{m \times n}$ be an arbitrary given matrix of rank $r \le \min\cbr{m,n}$, whose SVD is given by
\begin{align*}
    \Ab = \underset{m \times r}{\Ub_{A}}\ \underset{r \times r}{\Sigmab_{A}}\ \underset{r \times n}{\Vb_{A}^{\top}} = \sbr{\ub_{A,1},\dots,\ub_{A,r}} \diag\rbr{\sigma_{A,1},\dots,\sigma_{A,r}} \sbr{\vb_{A,1},\dots,\vb_{A,r}}^{\top}
\end{align*}
such that for any rank parameter $k \leq r$, we denote $\Ub_{A,k} \triangleq \sbr{\ub_{A,1},\dots,\ub_{A,k}}$ and $\Vb_{A,k} \triangleq \sbr{\vb_{A,1},\dots,\vb_{A,k}}$ as the orthonormal bases of the dimension-$k$ leading left and right singular subspaces of $\Ab$, while $\oc{\Ub_{A,k}} \triangleq\sbr{\ub_{A,k+1},\dots,\ub_{A,r}}$ and $\oc{\Vb_{A,k}} \triangleq\sbr{\vb_{A,k+1},\dots,\vb_{A,r}}$ as the orthonormal bases of the respective orthogonal complements. The diagonal submatrices consisting of the spectrum, $\Sigmab_{A,k} \triangleq \diag\rbr{\sigma_{A,1}, \dots, \sigma_{A,k}}$ and $\oc{\Sigmab_{A,k}} \triangleq \diag\rbr{\sigma_{A,k+1}, \dots, \sigma_{A,r}}$, follow analogously.
We denote $\Ab_k \triangleq \Ub_{A,k} \Sigmab_{A,k} \Vb_{A,k}^{\top}$ as the rank-$k$ truncated SVD that minimizes rank-$k$ approximation error of $\Ab$ (\cite{eckart1936}).
Furthermore, we denote the spectrum of $\Ab$, $\sigma\rbr{\Ab}$, as a $r \times r$ diagonal matrix, while for each $i=1,\dots,r$, let $\sigma_i(\Ab)$ be the $i$-th singular value of $\Ab$.

For the QR factorization, given an arbitrary rectangular matrices $\Mb \in \R^{d \times l}$ with full column rank ($d \geq l$), let $\Mb = \sbr{\Qb_{M},\oc{\Qb_{M}}} \sbr{\Rb_{M};\b{0}}$ be a full QR factorization of $\Mb$ such that $\Qb_{M} \in \R^{d \times l}$ and $\oc{\Qb_{M}} \in \R^{d \times (d-l)}$ consist of orthonormal bases of the subspace spanned by the columns of $\Mb$ and its orthogonal complement. We denote $\ortho: \csepp{\Mb \in\R^{d \times l}}{\rank\rbr{\Mb} = l} \to \R^{d \times l}$ ($d \geq l$) as a map that identifies an orthonormal basis (not necessary unique) for $\Mb$, $\ortho(\Mb) = \Qb_M$.

We adapt the MATLAB notation for matrices throughout this work. Unless specified otherwise ($\eg$, with subscripts), we use $\nbr{\cdot}$ to represent either the spectral norm or the Frobenius norm ($\ie$, holding simultaneously for both norms).

\subsection{Interpolative and CUR decompositions} \label{subsec:cur}
We first recall the definitions of the interpolative and CUR decompositions of a given $m\times n$ real matrix $\Ab$.
After providing the basic definitions, we discuss first how well it is theoretically possible to do
low-rank approximation under the constraint that natural bases must be used.
We then briefly describe the further suboptimality incurred by standard algorithms.

\subsubsection{Basic definitions}
We consider low-rank approximations for $\Ab$ with column and/or row subsets as bases. Given an arbitrary linearly independent column subset $\Cb = \Ab\rbr{:,J_s}$ ($J_s \subset [n]$), the rank-$\abbr{J_s}$ column ID of $\Ab$ with respect to column skeletons $J_s$ can be formulated as,
\begin{align}\label{eq:def_column_id}
    \cid{\Ab}{J_s} \triangleq \Cb \Cb ^{\pinv} \Ab,
\end{align}
where $\Cb \Cb ^{\pinv}$ is the orthogonal projector onto the spanning subspace of column skeletons.
Analogously, given any linearly independent row subset $\Rb = \Ab\rbr{I_s,:}$ ($I_s \subset [m]$), the rank-$\abbr{I_s}$ column ID of $\Ab$ with respect to row skeletons $I_s$ takes the form
\begin{align}\label{eq:def_row_id}
    \rid{\Ab}{I_s} \triangleq \Ab \Rb^{\pinv} \Rb,
\end{align}
where $\Rb^{\pinv} \Rb$ is the orthogonal projector onto the span of row skeletons.
While with both column and row skeletons, we can construct low-rank approximations for $\Ab$ in two forms -- the two-sided ID and CUR decomposition: with $\abbr{I_s}=\abbr{J_s}$, let $\Sb \triangleq \Ab\rbr{I_s,J_s}$ be an invertible two-sided skeleton of $\Ab$ such that
\begin{align}
    \label{eq:def_two_sided_id}
    &\t{Two-sided ID:}
    &&\tsid{\Ab}{I_s,J_s} \triangleq  \rbr{\Cb \Sb^{-1}} \Sb \rbr{\Cb^{\pinv} \Ab}
    \\
    \label{eq:def_cur}
    &\t{CUR decomposition:}
    &&\cur{\Ab}{I_s,J_s} \triangleq  \Cb \rbr{\Cb^\pinv \Ab \Rb^\pinv} \Rb
\end{align}
where in the exact arithmetic, since $\Sb^{-1} \Sb = \Ib$, the two sided ID is equivalent to the column ID characterized by $\Cb$, \ie, $\tsid{\Ab}{I_s,J_s} = \cid{\Ab}{J_s}$. Nevertheless, the two-sided ID $\tsid{\Ab}{I_s,J_s}$ and CUR decomposition $\cur{\Ab}{I_s,J_s}$ differ in both suboptimality and conditioning.

\begin{remark}[Suboptimality of ID versus CUR]\label{remark:id_cur_suboptimality}
For any given column and row skeletons $\Cb$ and $\Rb$, 
\begin{align}\label{eq:id_cur_suboptimality}
    \nbr{\Ab - \Cb\Cb^\pinv\Ab} \le \nbr{\Ab - \Cb\Cb^\pinv\Ab\Rb^\pinv\Rb} \le \rbr{\nbr{\Ab - \Cb\Cb^\pinv\Ab}^2 + \nbr{\Ab - \Ab\Rb^\pinv\Rb}^2}^{\frac{1}{2}}.
\end{align}
This comes from a simple orthogonal decomposition 
\begin{align*}
    \Ab - \Cb\Cb^\pinv\Ab\Rb^\pinv\Rb = \rbr{\Ib_m - \Cb\Cb^\pinv}\Ab + \Cb\Cb^\pinv\rbr{\Ab - \Ab\Rb^\pinv\Rb}
\end{align*}
where $\rbr{\Ib_m - \Cb\Cb^\pinv}$ and $\Cb\Cb^\pinv$ are orthogonal projectors. Therefore with the Frobenius norm,
\begin{align*}
    \nbr{\Ab - \Cb\Cb^\pinv\Ab\Rb^\pinv\Rb}_F^2 
    = &\nbr{\Ab - \Cb\Cb^\pinv\Ab}_F^2 + \nbr{\Cb\Cb^\pinv\rbr{\Ab - \Ab\Rb^\pinv\Rb}}_F^2 
    \\
    \le &\nbr{\Ab - \Cb\Cb^\pinv\Ab}_F^2 + \nbr{\Ab - \Ab\Rb^\pinv\Rb}_F^2,
\end{align*}    
while with the spectral norm
\begin{align*}
    &\nbr{\Ab - \Cb\Cb^\pinv\Ab\Rb^\pinv\Rb}_2^2  
    = \max_{\nbr{\vb}_2 \le 1} \nbr{\rbr{\Ab - \Cb\Cb^\pinv\Ab}\vb}_2^2 + \nbr{\Cb\Cb^\pinv\rbr{\Ab - \Ab\Rb^\pinv\Rb} \vb}_2^2
    \\
    &\begin{cases}
    \ge \max_{\nbr{\vb}_2 \le 1} \nbr{\rbr{\Ab - \Cb\Cb^\pinv\Ab}\vb}_2^2 = \nbr{\Ab - \Cb\Cb^\pinv\Ab}_2^2
    \\
    \le \max_{\nbr{\vb}_2 \le 1} \nbr{\rbr{\Ab - \Cb\Cb^\pinv\Ab}\vb}_2^2 + \nbr{\Cb\Cb^\pinv\rbr{\Ab - \Ab\Rb^\pinv\Rb} \vb}_2^2 
    \le \nbr{\Ab - \Cb\Cb^\pinv\Ab}_2^2 + \nbr{\Ab - \Ab\Rb^\pinv\Rb}_2^2
    \end{cases}.
\end{align*}    
\end{remark}

\begin{remark}[Conditioning of ID versus CUR]\label{remark:id_cur_stability}
The construction of CUR decomposition tends to be more ill-conditioned than that of two-sided ID. 
Precisely, for properly selected column and row skeletons $J_s$ and $I_s$, the corresponding skeletons $\Sb$, $\Cb$, and $\Rb$ shares similar spectrum decay as $\Ab$, which is usually ill-conditioned in the context.
In the CUR decomposition, both the bases $\Cb$, $\Rb$ and the small matrix $\Cb^\pinv \Ab \Rb^\pinv$ in the middle tend to suffer from large condition numbers as that of $\Ab$. 
In contrast, the only potentially ill-conditioned component in the two-sided ID is $\Sb$ ($\ie$, despite being expressed in $\Sb^{-1}$ and $\Cb^\pinv$, $\rbr{\Cb \Sb^{-1}}$ and $\rbr{\Cb^\pinv \Ab}$ in \Cref{eq:def_two_sided_id} are actually well-conditioned, and can be evaluated without direct inversions). 
\end{remark}

\begin{remark}[Stable CUR]\label{remark:stable_cur}
Numerically, the stable construction of a CUR decomposition $\cur{\Ab}{I_s,J_s}$ can be conducted via (unpivoted) QR factorization of $\Cb$ and $\Rb$ (\cite{anderson2015spectral}, Algorithm 2): let $\Qb_C \in \R^{m \times \abbr{J_s}}$ and $\Qb_R \in \R^{n \times \abbr{I_s}}$ be matrices from the QR whose columns form orthonormal bases for $\Cb$ and $\Rb^{\top}$, respectively, then
\begin{align}\label{eq:cur_stable}
    \cur{\Ab}{I_s,J_s} = \Qb_C \rbr{\Qb_C^{\top} \Ab \Qb_R} \Qb_R^{\top}.
\end{align}
\end{remark}    



\subsubsection{Notion of suboptimality}
\label{sec:cssp}

Both interpolative and CUR decompositions share the common goal of identifying proper column and/or row skeletons for $\Ab$ whose column and/or row spaces are well covered by the respective spans of these skeletons.
Without loss of generality, we consider the column skeleton selection problem: for a given rank $k < r$, we aims to find a proper column subset, $\Cb = \Ab\rbr{:,J_s}$ ($J_s \subset [n]$, $\abbr{J_s} = k$), such that
\begin{align}\label{eq:def_cssp_goal_close_to_opt_err}
    \nbr{\Ab - \cid{\Ab}{J_s}} \leq \phi\rbr{k,m,n} \nbr{\Ab - \Ab_k}
\end{align}
where common choices of the norm $\nbr{\cdot}$ include the spectral norm $\nbr{\cdot}_2$ and Frobenius norm $\nbr{\cdot}_F$; $\phi\rbr{k,m,n}$ is a function with $\phi\rbr{k,m,n} \geq 1$ for all $k,m,n$, and depends on the choice of $\nbr{\cdot}$; and we recall that $\Ab_k \triangleq \Ub_{A,k} \Sigmab_{A,k} \Vb_{A,k}^{\top}$ yields the optimal rank-$k$ approximation error. Meanwhile, similar low-rank approximation error bounds are desired for the row ID $\rid{\Ab}{I_s}$, two-sided ID $\tsid{\Ab}{I_s,J_s}$, and CUR decomposition $\cur{\Ab}{I_s,J_s}$.

\subsubsection{Suboptimality of matrix skeletonization algorithms}
The suboptimality of column subset selection, as well as the corresponding ID and CUR decomposition, has been widely studied in a variety of literature.
Specifically, \cite{goreinov1997} proved that with $\Sb = \Ab(I_s,J_s)$ ($\abbr{I_s}=\abbr{J_s}=k$) being the the maximal-volume submatrix in $\Ab$, the corresponding CUR decomposition (called pseudoskeleton component in the original paper) satisfies \Cref{eq:def_cssp_goal_close_to_opt_err} in $\nbr{\cdot}_2$ with $\phi = O\rbr{\sqrt{k}\rbr{\sqrt{m}+\sqrt{n}}}$. However, \cite{goreinov1997} pointed out that skeletons associated with the maximal-volume submatrix are not guaranteed to minimize the low-rank approximation error in \Cref{eq:def_cssp_goal_close_to_opt_err}.
Moreover, from the algorithmic perspective, identification of the maximal-volume submatrix is known to be NP-hard (\cite{civril2009}).
As the derivation of analysis for the strong rank-revealing QR factorization, \cite{gu1996} demonstrated the existence of a rank-$k$ column ID with $\phi = \sqrt{1+k(n-k)}$, and proposed a polynomial algorithm for constructing a relaxation with $\phi = O\rbr{ \sqrt{k(n-k)}}$.
Leveraging sampling based strategies, \cite{deshpande2006} showed the existence of a rank-$k$ column ID with $\phi=\sqrt{(k+1)(n-k)}$ for $\nbr{\cdot}_2$ and $\phi=\sqrt{1+k}$ for $\nbr{\cdot}_F$ by upper bounding the expectation of $\nbr{\Ab - \cid{\Ab}{J_s}}$ for volume sampling. Later on, \cite{deshpande2010, cortinovis2019lowrank} proposed polynomial-time algorithms for selecting such column skeletons.
The recent work \cite{derezinski2020improved} unveiled the multiple-descent trend of the suboptimality factor $\phi$ with respect to the approximation rank $k$, and illustrated that, depending on the spectrum decay, the suboptimality factor can be as tight as $\phi = O\rbr{k^{1/4}}$ for small $k$s, while for larger $k$s that fall in certain intervals, $\phi = \Omega(\sqrt{k})$.

\subsection{Randomized linear embeddings}\label{subsec:embedding}

For a given matrix $\Ab_k \in \R^{m \times n}$ of rank $k \leq \min\rbr{m,n}$ (typically we consider $k \ll \min(m,n)$), and a distortion parameter $\epsilon \in (0,1)$, a linear map $\Gammab: \R^m \to \R^l$ ($\ie$, $\Gammab \in \R^{l \times m}$, typically we consider $l \ll m$ for embeddings) is called an \textit{$\ell_2$ linear embedding} of $\Ab_k$ with distortion $\epsilon$ if
\begin{align}\label{eq:def_l2_rand_linear_embedding}
    (1 - \epsilon) \norm{\Ab_k\xb}_2 \leq \norm{\Gammab\Ab_k\xb}_2 \leq (1 + \epsilon) \norm{\Ab_k\xb}_2 \quad \forall\ \xb \in \R^n.
\end{align}
A distribution $\Scal$ over linear maps $\R^m \to \R^l$ (or equivalently, over $\R^{l \times m}$) generates \textit{randomized oblivious $\ell_2$ linear embeddings} (abbreviated as \textit{randomized linear embeddings}) if over $\Gammab \sim \Scal$, \Cref{eq:def_l2_rand_linear_embedding} holds for all $\Ab_k$ with at least constant probability. Given $\Ab_k$ and a randomized linear embedding $\Gammab \sim \Scal$, $\Gammab\Ab_k$ provides a (row) sketch of $\Ab_k$, and the process of forming $\Gammab\Ab_k$ is known as sketching (\cite{woodruff2015, martinsson2020}).

Randomized linear embeddings are closely related to various concepts like the Johnson-Lindenstrauss lemma and the restricted isometry property, and are studied in a broad scope of literature. \cite{martinsson2020} (Section 8, 9) provided an overview for and instantiated some popular choices of randomized linear embeddings, including
\begin{enumerate}
    \item Gaussian embedddings: $\Gammab \in \R^{l \times m}$ consisting of $\iid$ Gaussian entries $S_{ij} \sim \Ncal(\mu, 1/l)$ (\cite{indyk1998, halko2011, woodruff2015, martinsson2020});
    \item subsampled randomized trigonometric transforms (SRTT): $\Gammab = \sqrt{\frac{m}{l}} \Pib_{m \to l} \Tb \Phib \Pib_{m \to m}$ where
    $\Pib_{m \to l} \in \R^{l \times m}$ is a uniformly random selection of $l$ out of $m$ rows;
    $\Tb$ is an $m \times m$ unitary trigonometric transform ($\eg$, discrete Hartley transform for $\R$, and discrete Fourier transform for $\C$);
    $\Phib \triangleq \diag\rbr{\varphi_1,\dots,\varphi_m}$ with $\iid$ Rademacher random variables $\cbr{\varphi_i}_{i \in [m]}$ flips signs randomly; and
    $\Pib_{m \to m}$ is a random permutation (\cite{woolfe2008, halko2011, rokhlin2008, tropp2011, boutsidis2013srtt, martinsson2020}); and
    \item sparse sign matrices: $\Gammab = \sqrt{\frac{m}{\zeta}} \sbr{\sbb_1,\dots,\sbb_m}$ for some $2 \leq \zeta \leq l$, with $\iid$ $\zeta$-sparse columns $\cbr{\sbb_j \in \R^l}_{j \in [m]}$ constructed such that each $\sbb_j$ is filled with $\zeta$ independent Rademacher random variables at uniformly random coordinates (\cite{meng2013, nelson2013, woodruff2015, clarkson2017, tropp2017a, martinsson2020}).
\end{enumerate}
\Cref{tab:embedding_summary} summarizes lower bounds on $l$s that provide theoretical guarantee for \Cref{eq:def_l2_rand_linear_embedding}, along with asymptotic complexities of sketching, denoted as $T_s(l,\Ab_k)$, for these randomized linear embeddings.
In spite of the weaker guarantees for structured randomized embeddings ($\ie$, SRTTs and sparse sign matrices) in the theory by a logarithmic factor, from the empirical perpective, $l = \Omega\rbr{k/\epsilon^2}$ is usually sufficient for all the embeddings in \Cref{tab:embedding_summary} when considering tasks such as constructing randomized rangefinders (which we subsequently leverage for fast skeleton selection). For instance, \cite{halko2011, martinsson2020} suggested taking $l=k+\Omega(1)$ ($\eg$, $l=k+10$) for Gaussian embeddings, $l=\Omega\rbr{k}$ for SRTTs, and $l=\Omega\rbr{k}$, $\zeta = \min\rbr{l,8}$ (\cite{tropp2019streaming}) for sparse sign matrices in practice.

\begin{table}
    \centering
    \begin{tabular}{c|c|c}
    \hline
    Randomized linear embedding & Theoretical best dimension reduction & $T_s(l,\Ab_k)$ \\
    \hline
    Gaussian embedding & $l = \Omega\rbr{k/\epsilon^2}$ & $O(\nnz(\Ab_k) l)$ \\
    SRTT & $l = \Omega\rbr{k \log k / \epsilon^2}$ & $O(mn \log l)$ \\
    Sparse sign matrix & $l = \Omega\rbr{k \log k / \epsilon^2}$, $\zeta = \Omega\rbr{\log k / \epsilon}$ & $O(\nnz(\Ab_k) \zeta)$\\
    \hline
    \end{tabular}
    \caption{Lower bounds of $l$s that provide theoretical guarantee for \Cref{eq:def_l2_rand_linear_embedding}, and asymptotic complexities of sketching, $T_s(l,\Ab_k)$, for some common randomized linear embeddings.}
    \label{tab:embedding_summary}
\end{table}

\subsection{Randomized rangefinder and low-rank SVD}\label{subsec:rsvd}
Given $\Ab \in \R^{m \times n}$, the randomized rangefinder problem aims to construct a matrix $\Xb \in \R^{l \times n}$ such that the row space of $\Xb$ aligns well with the leading right singular subspace of $\Ab$ (\cite{martinsson2020}): $\nbr{\Ab - \Ab \Xb^{\pinv} \Xb}$ is sufficiently small for some unitary invariant norm $\nbr{\cdot}$ ($\eg$, $\nbr{\cdot}_2$ or $\nbr{\cdot}_F$). 
When $\Xb$ admits full row rank, we call $\Xb$ a rank-$l$ row space approximator of $\Ab$. The well-known optimality result from \cite{eckart1936} demonstrated that, for a fixed rank $k$, the optimal rank-$k$ row space approximator of $\Ab$ is given by its leading $k$ right singular vectors: $\nbr{\Ab - \Ab \Vb_{A,k} \Vb_{A,k}^{\top}}_F^2 = \nbr{\Ab - \Ab_k}_F^2 = \sum_{i=k+1}^{\min\cbr{m,n}} \sigma_i\rbr{\Ab}^2$.

A row sketch $\Xb = \Gammab \Ab$ generated by some proper randomized linear embedding $\Gammab$ is known to serve as a good solution for the randomized rangefinder problem with high probability. For instance, \cite{halko2011} demonstrated that, with the Gaussian embedding, a small constant oversampling $l-k \geq 4$ is sufficient for a good approximation:
\begin{align}\label{eq:rand_rangefinder_error_bound}
    \E \sbr{\nbr{\Ab - \Ab \Xb^{\pinv} \Xb}_F^2} \leq \frac{l-1}{l-k-1} \nbr{\Ab - \Ab_k}_F^2,
\end{align}
and moreover, $\nbr{\Ab - \Ab \Xb^{\pinv} \Xb}_F^2 \lesssim l (l-k) \log\rbr{l-k} \nbr{\Ab - \Ab_k}_F^2$ with high probability.
Similar guarantees hold for spectral norm (\cite{halko2011}, Section 10).
The randomized rangefinder error depends on the spectral decay of $\Ab$, and can be aggravated by a flat spectrum. In this scenario, power iterations (with proper orthogonalization, \cite{halko2011}), as well as Krylov and block Krylov subspace iterations (\cite{musco2015randomized}), may be incorporated after the initial sketching as a remedy. For example, with a randomized linear embedding $\Omegab$ of size $l \times n$, a row space approximator with $q$ power iterations ($q \geq 1$) is given by
\begin{align}\label{eq:def_power_iter}
    \Xb = \Omegab \rbr{\Ab^{\top} \Ab}^{q}
,\end{align}
and takes $\Xb$ takes $O\rbr{T_s(l,\Ab) + (2q-1) \nnz(\Ab) l}$ operations to construct. However, such plain power iteration in \Cref{eq:def_power_iter} is numerically unstable, and can lead to large errors when $\Ab$ is ill-conditioned and $q > 1$. For a stable construction, orthogonalization can be applied at each iteration:
\begin{align}\label{eq:ortho_power_iter}
    & \Yb^{(1)} = \Ab \Omegab^{\top} \nonumber
    \\
    & \Yb^{(i)} = \ortho\rbr{\Ab \ortho\rbr{\Ab^{\top} \Yb^{(i-1)}}} \ \forall\ i = 2,\dots,q\ (\t{if}\ q > 1) \nonumber
    \\
    & \Xb = \ortho\rbr{\Yb^{(q)}}^{\top} \Ab
\end{align}
with an additional cost of $O\rbr{q(m+n)l^2}$ overall.

In addition, with a proper $l$ that does not exceed the exact rank of $\Ab$, the row sketch $\Xb \in \R^{l \times n}$ has full row rank almost surely.
Precisely, recall $r=\rank(\Ab) \le \min\cbr{m,n}$ from \Cref{subsec:notation}:
\begin{lemma}\label{lemma:sketch_for_rand_rangefinder}
Let $\Gammab \in \R^{l \times m}$ be a randomized linear embedding with $\iid$ entries drawn from some absolutely continuous distribution (with respect to the Lebesgue measure). Then with $l \leq r$, the row sketch $\Xb = \Gammab \Ab$ has full row rank with probability $1$.
\end{lemma}
\begin{proof}[Proof of \Cref{lemma:sketch_for_rand_rangefinder}]
We first claim that all $l \times l$ submatrices in $\Gammab$ are invertible with probability $1$. It is sufficient to show that all $l \times l$ submatrices have nonzero determinants almost surely. The determinant of an $l \times l$ matrix can be expressed as a polynomial in $l^2$ entries of $\Gammab$, where zero set of the polynomial has Lebesgue measure $0$ in $\R^{l^2}$. Then since the distribution of entries in $\Gammab$ is absolutely continuous, the determinant is nonzero with probability $1$. Second, for $\Xb = \Gammab \Ab$ admitting full row rank, we observe that rows in $\Gammab$ can be interpreted as independent random vectors in $\R^m$ with $\iid$ entries from some absolutely continuous distribution. We aim to show that projection of the dimension-$l$ row space of $\Gammab$ onto the range of $\Ab$ remains dimension $l$ almost surely. But since any proper subspaces of $\R^m$ have Lebesgue measure $0$ in $\R^m$, projection of the row space of $\Gammab$ onto the range of $\Ab$ is dimension $l$, and therefore $\Xb$ admits full row rank, almost surely.
\end{proof}

A low-rank row space approximator $\Xb$ can be subsequently leveraged to construct a randomized rank-$l$ SVD.
Assuming $l$ is properly chosen such that $\Xb$ has full row rank, let $\Qb_X \in \R^{n \times l}$ be an orthonormal basis for the row space of $\Xb$.
\cite{halko2011} pointed out that the exact SVD of the smaller matrix $\Ab \Qb_X \in \R^{m \times l}$:
\begin{align}\label{eq:rsvd_procedures}
    \sbr{\underset{m \times l}{\wh \Ub_A}, \underset{l \times l}{\wh \Sigmab_A}, \underset{l \times l}{\wt \Vb_A}} = \t{svd} \rbr{\underbrace{\Ab \Qb_X}_{m \times l}, \t{`econ'}},
    \quad
    \underset{n \times l}{\wh\Vb_A} = \Qb_X \wt\Vb_A,
\end{align}
can be evaluated efficiently in $O\rbr{ml^2}$ operations (and $O\rbr{nl^2}$ additional operations for constructing $\wh\Vb_A$) such that $\Ab \approx \Ab \Xb^{\pinv} \Xb = \wh\Ub_A \wh\Sigmab_A \wh\Vb_A^{\top}$.

\subsection{Matrix decompositions with pivoting}\label{subsec:pivoting_mat_decomp}
We next briefly survey how pivoted QR and LU decompositions can be leveraged to
resolve the matrix skeleton selection problem.
In this section, $\Xb \in \R^{l \times n}$ denotes a matrix of full row rank
(that will typically arise as a ``row space approximator'').
Let $\Xb^{(t)} \in \R^{l \times n}$ be the resulted matrix after the $t$-th step of pivoting and matrix updating, so that $\Xb^{(0)} = \Xb$.

\subsubsection{Column pivoted QR (\pqr)}
Applying the \pqr to $\Xb$ gives:
\begin{align}\label{eq:def_QRCP}
    \Xb\ \underset{n \times n}{\Pib_n} = \Xb\ \underset{n \times l}{\left[\Pib_{n,1}\right.},\ \underset{n \times (n-l)}{\left.\Pib_{n,2}\right]} =  \underset{l \times l}{\Qb_l}\ \underset{l \times n}{\Rb^{QR}} = \Qb_l\ \underset{l \times l}{\left[\Rb^{QR}_{1}\right.},\ \underset{l \times (n-l)}{\left.\Rb^{QR}_{2}\right]},
\end{align}
where $\Qb_l$ is an orthogonal matrix; $\Rb^{QR}_{1}$ is upper triangular; and $\Pib_n \in \R^{n \times n}$ is a column permutation.
QR decompositions rank-$1$ update the active submatrix at each step for orthogonalization ($\eg$, \cite{householder1958hhqr}, \cite{trefethen1997}, Algorithm 10.1).
For each $t=0,\dots,l-2$, at the $(t+1)$-th step, the \pqr searches the entire active submatrix $\Xb^{(t)}\rbr{t+1:l, t+1:n}$ for the $(t+1)$-th column pivot with the maximal $\ell_2$-norm:
\begin{align*} 
    j_{t+1} = \argmax_{t+1 \leq j \leq n} \nbr{\Xb^{(t)}\rbr{t+1:l, j}}_2.
\end{align*}
\cite{gu1996} illustrated that \pqr satisfies $\max_{i,j}\abbr{\rbr{\rbr{\Rb^{QR}_{1}}^{-1} \Rb^{QR}_{2}}_{ij}} \leq 2^{l-i}$, and instantiated the classical Kahan matrix (\cite{kahan1966}) as a pathological case that admits exponential growth. Nevertheless, these adversarial inputs are scarce and sensitive to perturbations. The empirical success of \pqr also suggests that exponential growth with respect to $l$ almost never occurs in practice (\cite{trefethen1990}).
Meanwhile, there exist more sophisticated variations of \pqr, like the rank-revealing (\cite{chan1987, hong1992}) and strong rank-revealing QR (\cite{gu1996}), guaranteeing that $\max_{i,j} \abbr{\rbr{\rbr{\Rb^{QR}_{1}}^{-1} \Rb^{QR}_{2}}_{ij}}$ is upper bounded by some low-degree polynomial in $l$, but coming with higher complexities as trade-off.

\subsubsection{LU with partial pivoting (\plu)}
Applying the \plu columnwise to $\Xb$ yields:
\begin{align}\label{eq:def_LUPP}
    \Xb\ \underset{n \times n}{\Pib_n} = \Xb\ \underset{n \times l}{\left[\Pib_{n,1}\right.},\ \underset{n \times (n-l)}{\left.\Pib_{n,2}\right]} =  \underset{l \times l}{\Lb_l}\ \underset{l \times n}{\Rb^{LU}} = \Lb_l\ \underset{l \times l}{\left[\Rb^{LU}_{1}\right.},\ \underset{l \times (n-l)}{\left.\Rb^{LU}_{2}\right]},
\end{align}
where $\Lb_l$ is lower triangular; $\Rb^{LU}_{1}$ is upper triangular; $\Rb_1^{LU}\rbr{i,i} = 1$ and $\abbr{\Rb^{LU}\rbr{i,j}} \leq 1$ for all $i \in [l]$, $i \leq j \leq n$; and $\Pib_n \in \R^{n \times n}$ is a column permutation.
LU decompositions update active submatrices via Shur complements ($\eg$, \cite{trefethen1997}, Algorithm 21.1): for $t=0,\dots,l-2$,
\begin{align*}
    \Xb^{(t+1)}\rbr{t+2:l,t+2:n} = \Xb^{(t)}\rbr{t+2:l,t+2:n} - \Xb^{(t)}\rbr{t+2:l,t} \Xb^{(t)}\rbr{t,t+2:n}/\Xb^{(t)}\rbr{t,t}.
\end{align*}
At the $(t+1)$-th step, the (column-wise) \plu searches only the $(t+1)$-th row in the active submatrix, $\Xb^{(t)}\rbr{t+1:l, t+1:n}$, and pivots
\begin{align*}
    j_{t+1} = \argmax_{t+1 \leq j \leq n} \abbr{\Xb^{(t)}\rbr{t+1, j}},
\end{align*}
such that $\Rb^{LU}\rbr{i,j} = \Xb^{(i-1)}\rbr{i, j} / \Xb^{(i)}\rbr{i, i}$ for all $i \in [l]$, $i+1 \leq j \leq n$ (except for $\Rb^{LU}\rbr{i,j_i} = \Xb^{(i-1)}\rbr{i, i} / \Xb^{(i)}\rbr{i, i}$), and therefore $\abbr{\Rb^{LU}\rbr{i,j}} \leq 1$.

Analogous to \pqr, the pivoting strategy of \plu leads to a loose, exponential upper bound:
\begin{lemma}\label{lemma:LUPP_entry_bound}
The column-wise \plu in \Cref{eq:def_LUPP} satisfies that
\begin{align*}
    \max_{i,j} \abbr{\rbr{\rbr{\Rb^{LU}_{1}}^{-1} \Rb^{LU}_{2}}_{ij}} \leq 2^{l-i},
\end{align*}
where the upper bound is tight, for instance, when $\Rb^{LU}_{1}\rbr{i,j}=-1$ for all $i \in [l-1]$, $i+1 \leq j \leq l$ and $\Rb^{LU}_{2}\rbr{i,j} = 1$ for all $i \in [l]$, $j \in [n-l]$ ($\ie$, a Kahan-type matrix (\cite{kahan1966, peters1975})).
\end{lemma}

\begin{proof}[Proof of \Cref{lemma:LUPP_entry_bound}]
We start by observing the following recursive relations: for all $j=1,\dots,n-l$,
\begin{align*}
    & \rbr{\rbr{\Rb^{LU}_{1}}^{-1} \Rb^{LU}_{2}}_{lj} = \Rb^{LU}_{2}\rbr{l,j}
    \\
    & \rbr{\rbr{\Rb^{LU}_{1}}^{-1} \Rb^{LU}_{2}}_{ij} = \Rb^{LU}_{2}\rbr{i,j} - \sum_{\iota=i+1}^l \Rb^{LU}_{1}\rbr{i,\iota} \rbr{\rbr{\Rb^{LU}_{1}}^{-1} \Rb^{LU}_{2}}_{\iota,j}
    \quad \forall\ i=l-1,\dots,1,
\end{align*}
given $\Rb_1^{LU}\rbr{i,i} = 1$. Then both the upper bound and the Kahan example in \Cref{lemma:LUPP_entry_bound} follow from the fact that $\abbr{\Rb^{LU}\rbr{i,j}} \leq 1$ for all $i \in [l]$, $i \leq j \leq n$.
\end{proof}
In addition to the exponential worse-case scenario in \Cref{lemma:LUPP_entry_bound}, \plu is also vulnerable to rank deficiency since it only views one row for each pivoting step (in contrast to \pqr which searches the entire active submatrix). The advantage of the \plu type pivoting scheme is its superior empirical efficiency and parallelizability (\cite{geist1988lu, kurzak2007, grigori2011calu, solomonik2011communication}).
Fortunately, as with \pqr, adversarial inputs for \plu are sensitive to perturbations ($\eg$, flip the signs of random off-diagonal entries in $\Rb_1^{LU}$), and rarely encounted in practice.

\plu can be further stabilized with randomization (\cite{pan2015rand, pan2017num, trefethen1990}). In terms of the worse-case exponential entry-wise bound in \Cref{lemma:LUPP_entry_bound}, \cite{trefethen1990} investigated average-case growth factors of \plu on random matrices drawn from a variety of distributions ($\eg$, the Gaussian distribution, uniform distributions, Rademacher distribution, symmetry / Toeplitz matrices with Gaussian entries, and orthogonal matrices following Haar measure), and conjectured that the growth factor increases sublinearly with respect to the problem size in average cases.

\begin{remark}[Conjectured in \cite{trefethen1990}]\label{remark:lupp_stable_in_practice}
With randomized preprocessing like sketching, \plu is robust to adversarial inputs in practice, with $\max_{i,j} \abbr{\rbr{\rbr{\Rb^{LU}_{1}}^{-1} \Rb^{LU}_{2}}_{ij}} = O(l)$ in average cases.
\end{remark}

Some common alternatives to the partial pivoting for LU decompositions include (adaptive) cross approximations (\cite{Bebendorf2000ApproximationOB, tyrtyshnikov2000, zhao2005}), and complete pivoting.
Specifically, the complete pivoting is a more robust ($\eg$, to rank deficiency) alternative to partial pivoting that searches the entire active submatrix, and permutes rows and column simultaneously. In spite of lacking theoretical guarantees for the plain complete pivoting, like for QR decompositions, there exists modified complete pivoting strategies for LU that come with better rank-revealing guarantees (\cite{pan2000rrlu, gu2003srrlu,anderson2017, chen2020}), but higher computational cost as trade-off.

\section{Summary of existing algorithms}\label{sec:exist_cur_algo}

A vast assortment of algorithms for interpolative and CUR decompositions have been proposed and analyzed ($\eg$, \cite{deshpande2006, deshpande2010, drineas2008relative, mahoney2009,  bien2010cur, wang2012cur, cohen2014, woodruff2014optimalCUR, boutsidis2014, sorensen2014, anderson2015spectral, aizenbud2016randomized, voronin2017, anderson2017, shabat2018randomized, cortinovis2019lowrank, chen2020, derezinski2020improved}) in the past decades. From the skeleton selection perspection, these algorithms broadly fall in two categories:
\begin{enumerate}
    \item sampling based methods that draw matrix skeletons (directly, adaptively, or iteratively) from some proper distributions, and
    \item pivoting based methods that pick matrix skeletons greedily by constructing low-rank matrix decompositions with pivoting.
\end{enumerate}
In this section, we discuss existing algorithms for matrix skeletonizations, with a focus on algorithms based on randomized linear embeddings and matrix decompositions with pivoting.

\subsection{Sampling based skeleton selection}\label{subsec:sampling_cur}
The idea of skeleton selection via sampling is closely related to various topics including graph sparsification (\cite{batson2008}) and volume sampling (\cite{deshpande2006}).
Concerning volume sampling, \cite{deshpande2010, anderson2015spectral, cortinovis2019lowrank, derezinski2020improved} discussed adaptive sampling strategies that lead to matrix skeletons with close-to-optimal error guarantees, as discussed in \Cref{sec:cssp}.
\cite{drineas2008relative, mahoney2009, bien2010cur, drineas2012} pointed out the effect of matrix coherence, an inherited property of the matrix, on skeleton sampling, and proposed the idea of leverage score sampling for constructing CUR decompositions, as well as efficient estimations for the leverage scores.
Beginning with uniform sampling, \cite{cohen2014} provided extensive analysis on sampling based matrix approximations, and proposed an iterative sampling scheme for skeleton selection.
Some more sophisticated variations of sampling based skeleton selection algorithms were proposed in \cite{woodruff2014optimalCUR, boutsidis2014, wang2012cur} that combined varieties of sampling methods.

\subsection{Skeleton selection via deterministic pivoting}\label{subsec:pivoting_cur}
Greedy algorithms based on column and row pivoting can also be used for matrix skeletonizations. For instance, with proper rank-revealing pivoting like the strong rank-revealing QR proposed in \cite{gu1996}, a rank-$k$ ($k < r$) column ID can be constructed with the first $k$ column pivots
\begin{align*}
    \Ab \underset{n \times l}{\left[\Pib_{n,1}\right.},\ \underset{n \times (n-l)}{\left.\Pib_{n,2}\right]} = \underset{m \times k}{\left[\Qb_{A,1}\right.},\ \underset{m \times (r-k)}{\left.\Qb_{A,2}\right]} \bmat{\Rb_{A,11} & \Rb_{A,12} \\ \b{0} & \Rb_{A,22}} \approx \rbr{\Ab \Pib_{n,1}} \sbr{\Ib_k,\ \Rb_{A,11}^{-1} \Rb_{A,12}}
\end{align*}
where $\Pib_n = \sbr{\Pib_{n,1}, \Pib_{n,2}}$ is a permutation of columns; and $\Rb_{A,11}$ and $\Rb_{A,22}$ are non-singular and upper triangular. $\Cb = \rbr{\Ab \Pib_{n,1}}$ are the selected column skeletons that satisfies $\nbr{\Ab - \Cb \Cb^{\pinv} \Ab} = \nbr{\Rb_{A,22}} \lesssim \sqrt{k(n-k)} \nbr{\Ab - \Ab_k}$. As a more affordable alternative to the rank-revealing pivoting, the \pqr discussed in \Cref{subsec:pivoting_mat_decomp} also works well for skeleton selection in practice (\cite{voronin2017}), despite the weaker theoretical guarantee ($\ie$, $\nbr{\Ab - \Cb \Cb^{\pinv} \Ab} \lesssim 2^k \nbr{\Ab - \Ab_k}$) due to the known existence of adversarial inputs.
In addition to the QR based pivoting schemes, \cite{pan2000rrlu, gu2003srrlu, aizenbud2016randomized, anderson2017, shabat2018randomized, chen2020} proposed various (randomized) LU based pivoting algorithms with rank-revealing guarantees that can be leveraged for greedy matrix skeleton selection (as discussed in \cref{subsec:pivoting_mat_decomp}). Meanwhile, relaxing the rank-revealing requirements on pivoting schemes by pre-processing $\Ab$, \cite{sorensen2014, drmac2016qdeim} proposed the DEIM skeleton selection algorithm that apply \plu on the leading singular vectors of $\Ab$.

\subsection{Randomized pivoting based skeleton selection}\label{subsec:rand_pivot_cur}
In comparison to the sampling based skeleton selection, the deterministic pivoting based skeleton selection methods suffer from two major drawbacks. First, pivoting is usually unaffordable for large-scale problems in common modern applications. Second, classical pivoting schemes like the \pqr and \plu are vulnerable to antagonistic inputs.
Fortunately, randomized pre-processing with sketching provides remedies both problems:
\begin{enumerate}
    \item Faster execution speed is attained by executing classical pivoting schemes on a sketch
    $\Xb = \Gammab\Ab \in \R^{l \times n}$, for some randomized embedding $\Gammab$, instead on
    $\Ab$ directly.
    \item With randomization, classical pivoting schemes like the \pqr and \plu 
    are robust to adversarial inputs in practice (\Cref{remark:lupp_stable_in_practice}, \cite{trefethen1990}).
\end{enumerate}
\Cref{algo:sketch_pivot_CUR_general} describes a general framework for randomized pivoting based skeleton selection. Grounding this framework down with different combinations of row space approximators and pivoting schemes, \cite{voronin2017} took $\Xb = \Gammab \Ab$ as a row sketch, and applied \pqr to $\Xb$ for column skeleton selection.
Alternatively, the DEIM skeleton selection algorithm proposed in \cite{sorensen2014} can be accelerated by taking $\Xb$ as an approximation of the leading-$l$ right singular vectors of $\Ab$ (\Cref{eq:rsvd_procedures}), where \plu is applied for skeleton selection.

\begin{algorithm}[ht]
\caption{Randomized pivoting based skeleton selection: a general framework}\label{algo:sketch_pivot_CUR_general}
\begin{algorithmic}[1]
\REQUIRE $\Ab \in \R^{m \times n}$ of rank $r$, rank $l \leq r$ (typically $l \ll \min(m,n)$).
\ENSURE Column and/or row skeleton indices, $J_s \subset [n]$ and/or $I_s \subset [m]$, $\abbr{J_s} = \abbr{I_s} = l$.

\STATE Draw an oblivious $\ell_2$-embedding $\Gammab \in \R^{l \times m}$.
\STATE Construct a row space approximator $\Xb \in R^{l \times n}$ via sketching with $\Gammab$.\\
$\eg$, $\Xb$ can be 1) a row sketch or 2) approximations of right singular vectors.
\STATE Perform column-wise pivoting on $\Xb$. Let $J_s$ index the $l$ column pivots.
\STATE Perform row-wise pivoting on $\Cb = \Ab(:,J_s)$. Let $I_s$ index the $l$ row pivots.
\end{algorithmic}
\end{algorithm}

First, we recall from \Cref{lemma:sketch_for_rand_rangefinder} that when taking $\Xb$ as a row sketch, with Gaussian embeddings, $\Gammab$ and $\Xb = \Gammab \Ab$ are both full row rank with probability $1$. 
Moreover, when taking $\Xb$ as an approximation of right singular vectors constructed with a row space approximator consisting of $l$ linearly independent rows, $\Xb$ also admits full row rank.

Second, when both column and row skeletons are inquired, \Cref{algo:sketch_pivot_CUR_general} selects the column skeletons first with randomized pivoting, and subsequently identifies the row skeletons by pivoting on the selected columns. 
With $\Xb$ being full row rank (almost surely when $\Gammab$ is a Gaussian embedding), the column skeletons $\Cb$ are linearly independent. Therefore, the row-wise skeletonization of $\Cb$ is exact, without introducing additional error. That is, the two-sided ID constructed by \Cref{algo:sketch_pivot_CUR_general} is equal to the associated column ID in exact arithmetic, $\tsid{\Ab}{I_s,J_s} = \cid{\Ab}{J_s}$.

\section{A simple but effective modification: \plu on sketches}\label{sec:cur_rand_lupp}
Inspired by the idea of pivoting on sketches (\cite{voronin2017}) and the remarkably competitive performance of \plu when applied to leading singular vectors (\cite{sorensen2014}), we propose a simple but effective modification -- applying \plu directly to a sketch of $\Ab$. In terms of the general framework in \Cref{algo:sketch_pivot_CUR_general}, this corresponds to taking $\Xb$ as a row sketch of $\Ab$, and then selecting skeletons via \plu on $\Xb$ and $\Cb$, as summarized in \Cref{algo:rand-id-lupp}.

\begin{algorithm}[ht]
\caption{Randomized \plu skeleton selection}
\label{algo:rand-id-lupp}
\begin{algorithmic}[1]
\REQUIRE $\Ab \in \R^{m \times n}$ of rank $r$, rank $l \leq r$ (typically $l \ll \min(m,n)$).
\ENSURE Column and/or row skeleton indices, $J_s \subset [n]$ and/or $I_s \subset [m]$, $\abbr{J_s} = \abbr{I_s} = l$.

\STATE Draw an oblivious $\ell_2$-embedding $\Gammab \in \R^{l \times m}$.
\STATE Construct a row sketch $\Xb = \Gammab \Ab$.
\STATE Perform column-wise \plu on $\Xb$. Let $J_s$ index the $l$ column pivots.
\STATE Perform row-wise \plu on $\Cb = \Ab(:,J_s)$. Let $I_s$ index the $l$ row pivots.
\end{algorithmic}
\end{algorithm}

Comparing to pivoting with \pqr (\cite{voronin2017}), \Cref{algo:rand-id-lupp} with \plu is empirically faster, as discussed in \Cref{subsec:pivoting_mat_decomp}, and illustrated in \Cref{fig:pivot-ps-time}.
Meanwhile, assuming that the true SVD of $\Ab$ is unavailable, in comparison to pivoting on the approximated leading singular vectors (\cite{sorensen2014}) from \Cref{eq:rsvd_procedures}, \Cref{algo:rand-id-lupp} saves the effort of constructing randomized SVD which takes $O\rbr{\nnz(\Ab)l + (m+n)l^2}$ additional operations. Additionally, with randomization, the stability of \plu conjectured in \cite{trefethen1990} (\Cref{remark:lupp_stable_in_practice}) applies, and \Cref{algo:rand-id-lupp} effectively circumvents the potential vulnerability of \plu to adversarial inputs in practice.
A formal error analysis of \Cref{algo:sketch_pivot_CUR_general} in general reflects these points:

\begin{theorem}[Column skeleton selection by pivoting on a row space approximator]\label{thm:pivoting_on_rangeapprox_error_bound}
Given a row space approximator $\Xb \in \R^{l \times n}$ ($l \leq r$) of $\Ab$ that adimits full row rank, let $\Pib_n \in \R^{n \times n}$ be the resulted permutation after applying some proper column pivoting scheme on $\Xb$ that identifies $l$ linearly independent column pivots: for the $(l,n-l)$ column-wise partition $\Xb \Pib_n = \Xb \sbr{\Pib_{n,1}, \Pib_{n,2}} = \sbr{\Xb_{1}, \Xb_{2}}$, the first $l$ column pivots $\Xb_1 = \Xb\Pib_{n,1} \in \R^{l \times l}$ admits full column rank. 
Moreover, the rank-$l$ column ID $\cid{\Ab}{J_s} = \Cb \Cb^{\pinv} \Ab$, with linearly independent column skeletons $\Cb = \Ab \Pib_{n,1}$, satisfies that
\begin{align}\label{eq:pivoting_on_rangeapprox_error_bound}
    \nbr{\Ab - \Cb \Cb^{\pinv} \Ab} \leq \eta \nbr{\Ab - \Ab \Xb^{\pinv} \Xb},
\end{align}
where $\eta \leq \sqrt{1 + \nbr{\Xb_{1}^{\pinv} \Xb_{2}}_2^2}$, and $\nbr{\cdot}$ represents the spectral or Frobenius norm.
\end{theorem}

\Cref{thm:pivoting_on_rangeapprox_error_bound} states that when selecting column skeletons by pivoting on a row space approximator, the low-rank approximation error of the resulting column ID is upper bounded by that of the associated row space approximator up to a factor $\eta > 1$ that can be computed a posteriori efficently in $O\rbr{l^2(n-l)}$ operations.
\Cref{eq:pivoting_on_rangeapprox_error_bound} essentially decouples the error from the row space approximation with $\Xb$ ($\nbr{\Ab - \Ab \Xb^{\pinv} \Xb}$ corresponding to Line 1 and 2 of \Cref{algo:sketch_pivot_CUR_general}) and that from the skeleton selection by pivoting on $\Xb$ ($\eta$ corresponding to Line 3 and 4 of \Cref{algo:sketch_pivot_CUR_general}).
Now we ground \Cref{thm:pivoting_on_rangeapprox_error_bound} with different choices of row space approximation and pivoting strategies:
\begin{enumerate}
    \item With \Cref{algo:rand-id-lupp}, $\nbr{\Ab - \Ab \Xb^{\pinv} \Xb}$ is the randomized rangefinder error (\Cref{eq:rand_rangefinder_error_bound}, \cite{halko2011} Section 10), and $\eta \leq \sqrt{1+\nbr{\rbr{\Rb_1^{LU}}^{-1} \Rb_2^{LU}}_2^2}$ (recall \Cref{eq:def_LUPP}). Although in the worse case scenario (where the entry-wise upper bound in \Cref{lemma:LUPP_entry_bound} is tight), $\eta = \Theta\rbr{2^l \sqrt{n-l}}$, with a randomized row sketch $\Xb$, assuming the stability of \plu conjectured in \cite{trefethen1990} holds (\Cref{remark:lupp_stable_in_practice}), $\eta = O\rbr{l^{3/2} \sqrt{n-l}}$.
    \item Skeleton selection with \pqr on row sketches ($\ie$, randomized \pqr proposed by \cite{voronin2017}) shares the same error bound as \Cref{algo:rand-id-lupp} ($\ie$, analogous arguments hold for $\nbr{\rbr{\Rb_1^{QR}}^{-1} \Rb_2^{QR}}_2^2$).
    \item When applying \plu on the true leading singular vectors ($\ie$, DEIM proposed in \cite{sorensen2014}, assuming that the true SVD is available), $\nbr{\Ab - \Ab \Xb^{\pinv} \Xb} = \nbr{\Ab - \Ab_l}$, but without randomization, LUPP is vulnerable to adversarial inputs which can lead to $\eta = \Theta\rbr{2^l \sqrt{n-l}}$ in the worse case.
    \item When applying \plu on approximations of leading singular vectors (constructed via \Cref{eq:rsvd_procedures}, $\ie$, randomized DEIM suggested by \cite{sorensen2014}), $\nbr{\Ab - \Ab \Xb^{\pinv} \Xb}$ corresponds to the randomized rangefinder error with power itertions (\cite{halko2011} Corollary 10.10), while $\eta$ follows the analogous analysis as for \Cref{algo:rand-id-lupp}.
\end{enumerate}

\begin{proof}[Proof of \Cref{thm:pivoting_on_rangeapprox_error_bound}]
We start by defining two oblique projectors
\begin{align*}
    \underset{n \times n}{\pobx} \triangleq \Pib_{n,1} \rbr{\Xb \Pib_{n,1}}^{\pinv} \Xb,
    \quad
    \underset{n \times n}{\pobc} \triangleq \Pib_{n,1} \rbr{\Cb^{\top} \Cb}^{\pinv} \Cb^{\top} \Ab,
\end{align*}
and observe that, since $\Cb$ consists of linearly independent columns, $\rbr{\Cb^{\top} \Cb}^{\pinv} \Cb^{\top} \Ab \Pib_{n,1} = \Ib_l$, and
\begin{align*}
    \pobc\pobx = \Pib_{n,1} \rbr{\Cb^{\top} \Cb}^{\pinv} \Cb^{\top} \Ab \Pib_{n,1} \rbr{\Xb \Pib_{n,1}}^{\pinv} \Xb = \pobx.
\end{align*}
With $\pobc$, we can express the column ID as
\begin{align*}
    \cid{\Ab}{J_s} = \Cb \Cb^{\pinv} \Ab = \Ab \rbr{\Pib_{n,1} \rbr{\Cb^{\top} \Cb}^{\pinv} \Cb^{\top} \Ab} = \Ab \pobc,
\end{align*}
Therefore, the low-rank approximation error of $\cid{\Ab}{J_s}$ satisfies
\begin{align*}
    \nbr{\Ab - \Cb \Cb^{\pinv} \Ab} = & \nbr{\Ab \rbr{\Ib - \pobc}}
    \\
    = & \nbr{\Ab \rbr{\Ib_n - \pobc} \rbr{\Ib_n - \pobx}}
    \\
    = & \nbr{\rbr{\Ib_m - \Cb \Cb^{\pinv}} \Ab \rbr{\Ib_n - \pobx}}
    \\
    \leq & \nbr{\Ib_m - \Cb \Cb^{\pinv}}_2 \nbr{\Ab \rbr{\Ib_n - \pobx}},
\end{align*}
where $\nbr{\Ib_m - \Cb \Cb^{\pinv}}_2 = 1$, and since $\Xb \pobx = \Xb_1 \Xb_1^{\pinv} \Xb = \Xb$ with $\Xb_1$ being full-rank,
\begin{align*}
    \nbr{\Ab \rbr{\Ib_n - \pobx}} = \nbr{\Ab \rbr{\Ib_n - \Xb^{\pinv} \Xb} \rbr{\Ib_n - \pobx}} = \nbr{\Ib_n - \pobx}_2 \nbr{\Ab \rbr{\Ib_n - \Xb^{\pinv} \Xb}}.
\end{align*}
As a result, we have
\begin{align*}
    \nbr{\Ab - \Cb \Cb^{\pinv} \Ab} \leq \nbr{\Ib_n - \pobx}_2 \nbr{\Ab \rbr{\Ib_n - \Xb^{\pinv} \Xb}},
\end{align*}
and it is sufficient to show that $\eta \triangleq \nbr{\Ib_n - \pobx}_2 \leq \sqrt{1 + \nbr{\Xb_{1}^{\pinv} \Xb_{2}}_2^2}$.
Indeed,
\begin{align*}
    \Ib_n - \pobx = \Pib_n^{\top} \Pib_n - \Pib_n^{\top} \Pib_{n,1} \rbr{\Xb \Pib_{n,1}}^{\pinv} \Xb \Pib_n
    = \Ib_n - \bmat{\Ib_l \\ \b0} \Xb_1^{\pinv} \bmat{\Xb_1 \Xb_2} = \bmat{\b0 & - \Xb_1^{\pinv} \Xb_2 \\ \b0 & \Ib_{n-l}}
\end{align*}
such that $\eta = \nbr{\sbr{- \Xb_1^{\pinv} \Xb_2; \Ib_{n-l}}}_2 \leq \sqrt{1 + \nbr{\Xb_{1}^{\pinv} \Xb_{2}}_2^2}$.
\end{proof}

Here, the proof of \Cref{thm:pivoting_on_rangeapprox_error_bound} is reminiscent of \cite{sorensen2014}, while it generalizes the result for fixed right leading singular vectors to any proper row space approximators of $\Ab$ ($\eg$, $\wh\Vb_A$ in \Cref{eq:rsvd_procedures}, or simply a row sketch). The generalization of \Cref{thm:pivoting_on_rangeapprox_error_bound} leads to a factor $\eta$ that is efficiently computable a posteriori, which can serve as an empirical replacement of the exponential upper bound induced by the scarce adversarial inputs.

In addition to empirical efficiency and robustness discussed above, \Cref{algo:rand-id-lupp} has another potential advantage: the skeleton selection algorithm can be easily adapted to the streaming setting. 
The streaming setting considers $\Ab$ as a data stream that can only be accessed as a sequence of snapshots. Each snapshot of $\Ab$ can be viewed only once, and the storage of the entire matrix $\Ab$ is infeasible (\cite{tropp2017a, tropp2019streaming, martinsson2020}).

\begin{algorithm}[ht]
\caption{Streaming \plu/\pqr skeleton selection}\label{algo:sketch_pivot_CUR_online}
\begin{algorithmic}[1]
\REQUIRE $\Ab \in \R^{m \times n}$ of rank $r$, rank $l \leq r$ (typically $l \ll \min(m,n)$).
\ENSURE Column and/or row skeleton indices, $J_s \subset [n]$ and/or $I_s \subset [m]$, $\abbr{J_s} = \abbr{I_s} = l$.

\STATE Draw independent oblivious $\ell_2$-embeddings $\Gammab \in \R^{l \times m}$ and $\Omegab \in \R^{l \times n}$.
\STATE Construct row and column sketches, $\Xb = \Gammab \Ab$ and $\Yb = \Ab \Omegab^{\top}$, in a single pass through $\Ab$.
\STATE Perform column-wise pivoting (\plu/\pqr) on $\Xb$. Let $J_s$ index the $l$ column pivots.
\STATE Perform row-wise pivoting (\plu/\pqr) on $\Yb$. Let $I_s$ index the $l$ row pivots.
\end{algorithmic}
\end{algorithm}

\begin{remark}
When only the column and/or row skeleton \textit{indices} are required (and not the explicit construction of the corresponding interpolative or CUR decomposition), \Cref{algo:rand-id-lupp} can be adapted to the streaming setting (as shown in \Cref{algo:sketch_pivot_CUR_online}) by sketching both sides of $\Ab$ independently in a single pass, and pivoting on the resulting column and row sketches.
Moreover, with the column and row skeletons $J_s$ and $I_s$ from \Cref{algo:sketch_pivot_CUR_online}, \Cref{thm:pivoting_on_rangeapprox_error_bound} and its row-wise analog together, along with \Cref{eq:id_cur_suboptimality}, imply that
\begin{align*}
    \nbr{\Ab - \cur{\Ab}{I_s,J_s}} \leq \eta_X \nbr{\Ab - \Ab \Xb^{\pinv} \Xb} + \eta_Y \nbr{\Ab - \Yb \Yb^{\pinv} \Ab},
\end{align*}
where $\eta_X$ and $\eta_Y$ are small in practice with pivoting on randomized sketches, and have efficiently a posteriori computable upper bounds given $\Xb$ and $\Yb$, as discussed previously. $\nbr{\Ab - \Ab \Xb^{\pinv} \Xb}_F^2$ and $\nbr{\Ab - \Yb \Yb^{\pinv} \Ab}$ are the randomized rangefinder errors with well-established upper bounds (\cite{halko2011} Section 10).
\end{remark}

We point out that, although the column and row skeleton selection can be conducted in streaming fashion, the explicit stable construction of ID or CUR requires two additional passes through $\Ab$: one pass for retrieving the skeletons $\Cb$ and/or $\Rb$, and the other pass to construct $\Cb^\pinv \Ab \Rb^\pinv$ for CUR, or $\Cb^\pinv \Ab$, $\Ab \Rb^\pinv$ for IDs.
In practice, for efficient estimations of the ID or CUR when revisiting $\Ab$ is expensive, it is possible to circumvent the second pass through $\Ab$ with compromise on accuracy and stability, albeit the inevitability of the first pass for skeleton retrieval. 

Precisely for the ID, $\Cb^\pinv \Ab$ (in \Cref{eq:def_column_id}) or $\Ab \Rb^\pinv$ (in \Cref{eq:def_row_id}) can be estimated without revisiting $\Ab$ leveraging the associated row and column sketches: 
\begin{align*}
    \Cb^\pinv \Ab \approx \Xb_1^\pinv \Xb,
    \quad 
    \Ab \Rb^\pinv \approx \Yb \Yb_1^\pinv,    
\end{align*}    
where $\Xb_1 = \Xb\rbr{:,J_s}$ and $\Yb_1 = \Yb\rbr{I_s,:}$ are the $l$ column and row pivots in $\Xb$ and $\Yb$, respectively.
Meanwhile for the CUR, by retrieving the skeletons $\Sb = \Ab\rbr{I_s,J_s}$, $\Cb=\Ab\rbr{:,J_s}$ and $\Rb=\Ab\rbr{I_s,:}$, we can construct a CUR decomposition $\Cb \Sb^{-1} \Rb$, despite the compromise on both accuracy and stability.

\section{Numerical experiments}\label{sec:experiments}

In this section, we study the empirical performance of various randomized skeleton selection algorithms.
Starting with the randomized pivoting based algorithms, we investigate the efficiency of two major components of \Cref{algo:sketch_pivot_CUR_general}: 
(1) the sketching step for row space approximator construction, and 
(2) the pivoting step for greedy skeleton selection. 
Then we explore the suboptimality (in terms of low-rank approximation errors of the resulting CUR decompositions $\nbr{\Ab - \cur{\Ab}{I_s,J_s}}$), as well as the efficiency (in terms of empirical run time), of different randomized skeleton selection algorithms.

We conduct all the experiments, except for those in \Cref{fig:sketch-ps-time} on the efficiency of sketching, in MATLAB R2020a. In the implementation, the computationally dominant processes, including the sketching, \plu, \pqr, and SVD, are performed by the MATLAB built-in functions.
The experiments in \Cref{fig:sketch-ps-time} are conducted in Julia Version 1.5.3 with the JuliaMatrices/LowRankApprox.jl package (\cite{lowrankjl2020}).

\subsection{Computational speeds of different embeddings}
\label{subsec:exp_embedding}
Here, we compare the empirical efficiency of constructing sketches with some common randomized embeddings listed in \Cref{tab:embedding_summary}. We consider applying an embedding $\Gammab$ of size $l \times m$ to a matrix $\Ab$ of size $m \times n$, which can be interpreted as embedding $n$ vectors in an ambient space $\R^m$ to a lower dimensional space $\R^l$. We scale the experiments with respect to the ambient dimension $m$, at several different embedding dimension $l$, with a fixed number of repetitions $n=1000$.
Figure \ref{fig:sketch-ps-time} suggests that, with proper implementation, the sparse sign matrices are more efficient than the Gaussian embeddings and the SRTTs, especially for large scale problems. The SRTTs outperform Gaussian embeddings in terms of efficiency, and such advantage can be amplified as $l$ increases. These observations align with the asymptotic complexity in \Cref{tab:embedding_summary}. While we also observe that, with MATLAB default implementation, the Gaussian embeddings usually enjoy matching efficiency as sparse sign matrices for moderate size problems, and are more efficient than SRTTs.

\begin{figure}[ht]
    \centering
    \begin{subfigure}{0.32\textwidth}
    \centering
    \includegraphics[width=\linewidth]{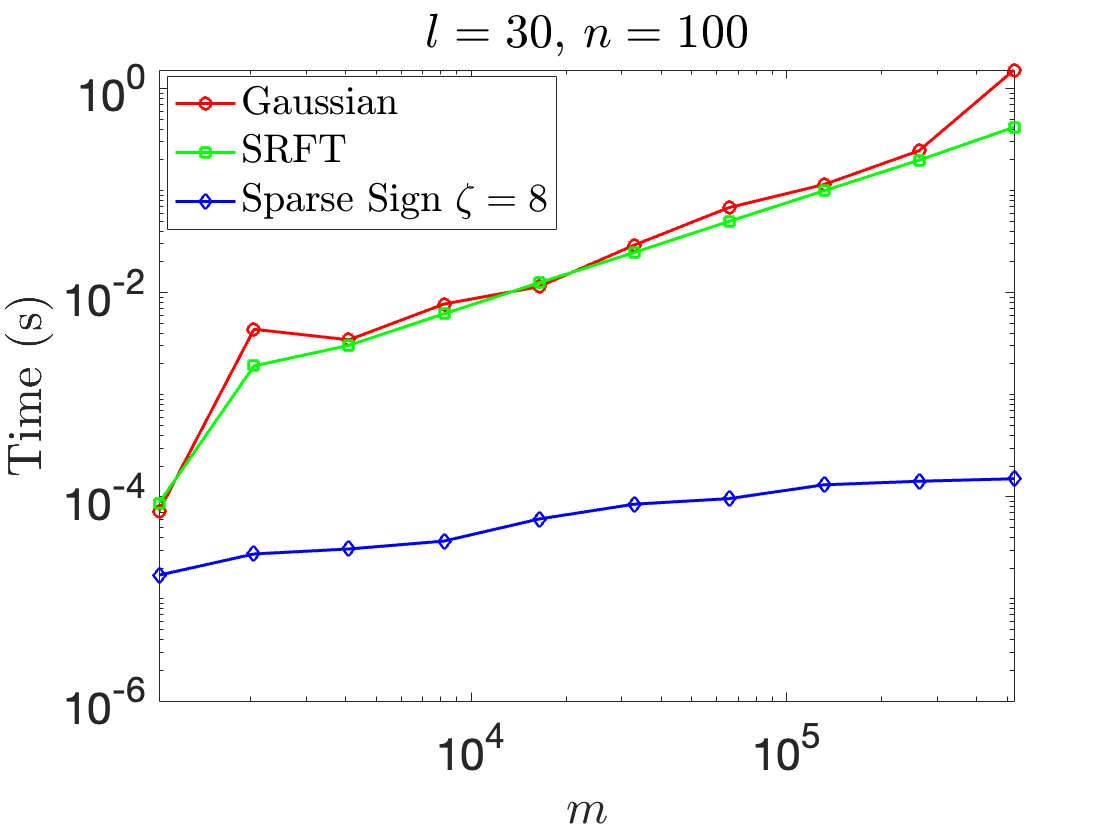}
    \end{subfigure}
    \begin{subfigure}{0.32\textwidth}
    \centering
    \includegraphics[width=\linewidth]{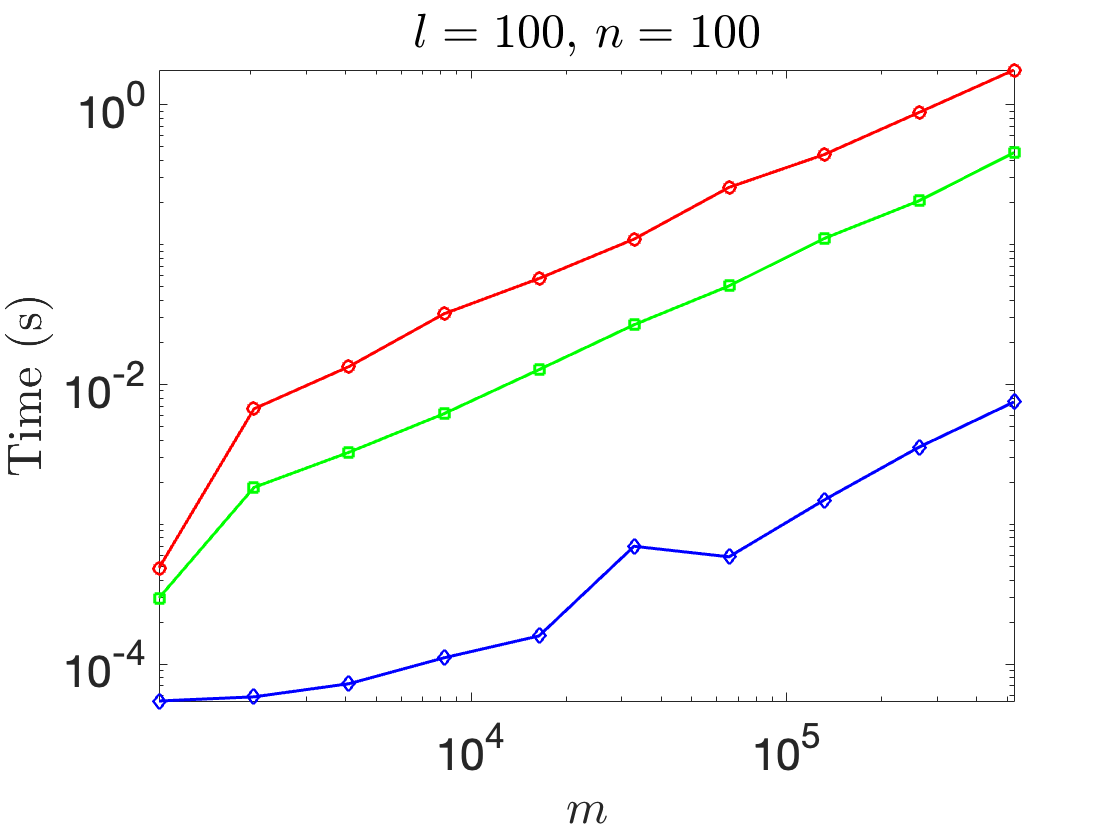}
    \end{subfigure}
    \begin{subfigure}{0.32\textwidth}
    \centering
    \includegraphics[width=\linewidth]{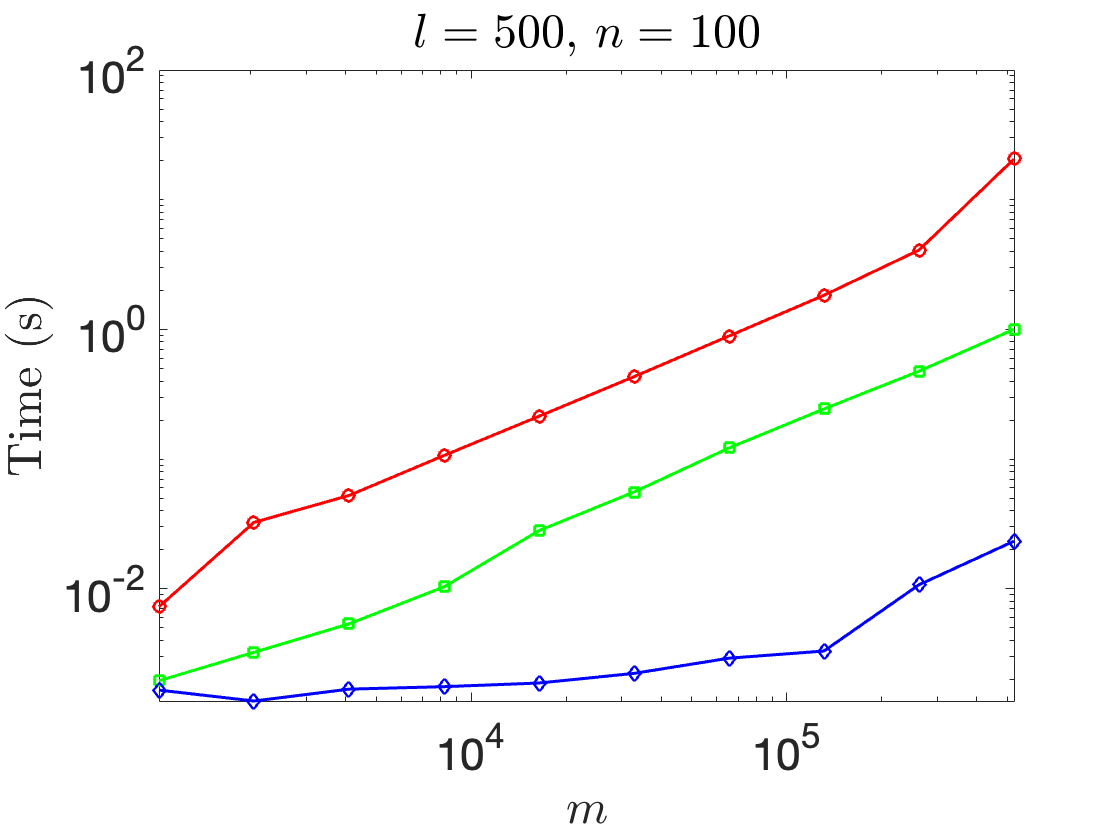}
    \end{subfigure}
    \caption{Run time of applying different randomized embeddings $\bs{\Gamma} \in \R^{l \times m}$ to some dense matrices of size $m \times n$, scaled with respect to the ambience dimension $m$, with different embedding dimension $l$, and a fixed number of embeddings $n=100$.}
    \label{fig:sketch-ps-time}
\end{figure}

\subsection{Computational speeds of different pivoting schemes}\label{subsec:exp_pivot}
Given a sketch of $\Ab$, we isolate different pivoting schemes in \Cref{algo:sketch_pivot_CUR_general}, and compare their run time as the problem size $n$ increases.
Specifically, the \plu and \pqr pivot directly on the given row sketch $\Xb = \Gammab \Ab \in \R^{l \times n}$, while the DEIM involves one additional power iteration with orthogonalization (\Cref{eq:ortho_power_iter}) before applying the \plu ($\ie$, with a given column sketch $\Yb = \Ab \Omegab \in \R^{m \times l}$, for DEIM, we first construct an orthonormal basis $\Qb_Y \in \R^{m \times l}$ for columns of the sketch, and then we compute the reduced SVD for $\Qb_Y^{\top} \Ab \in l \times n$, and finally we column-wisely pivot on the resulting right singular vectors of size $l \times n$).
\begin{figure}[!ht]
    \centering
    \begin{subfigure}{0.32\textwidth}
    \centering
    \includegraphics[width=\linewidth]{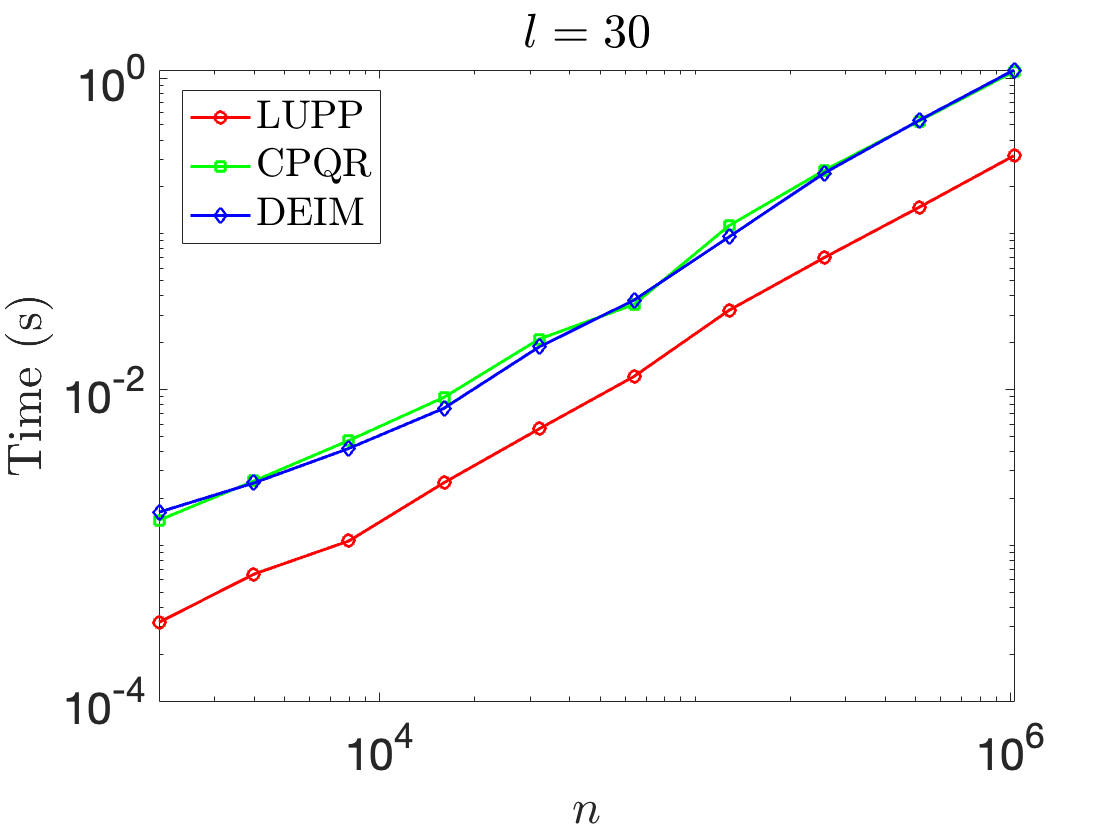}
    \end{subfigure}
    \begin{subfigure}{0.32\textwidth}
    \centering
    \includegraphics[width=\linewidth]{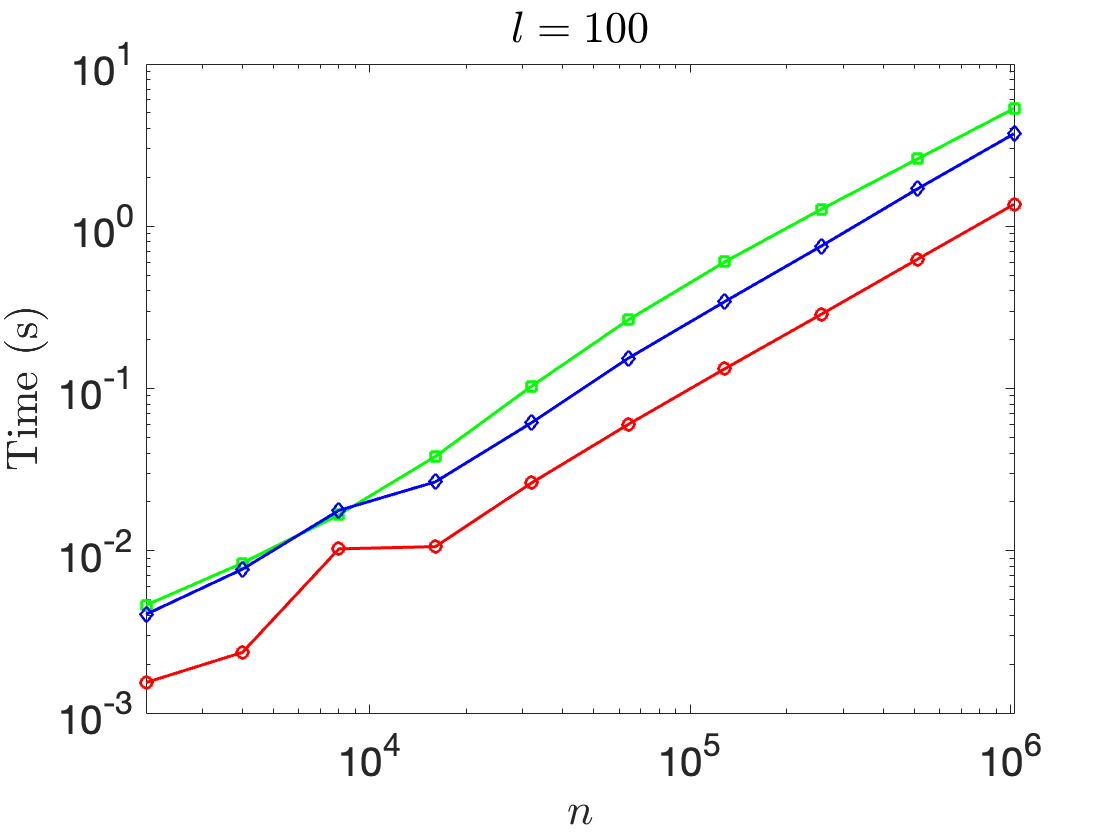}
    \end{subfigure}
    \begin{subfigure}{0.32\textwidth}
    \centering
    \includegraphics[width=\linewidth]{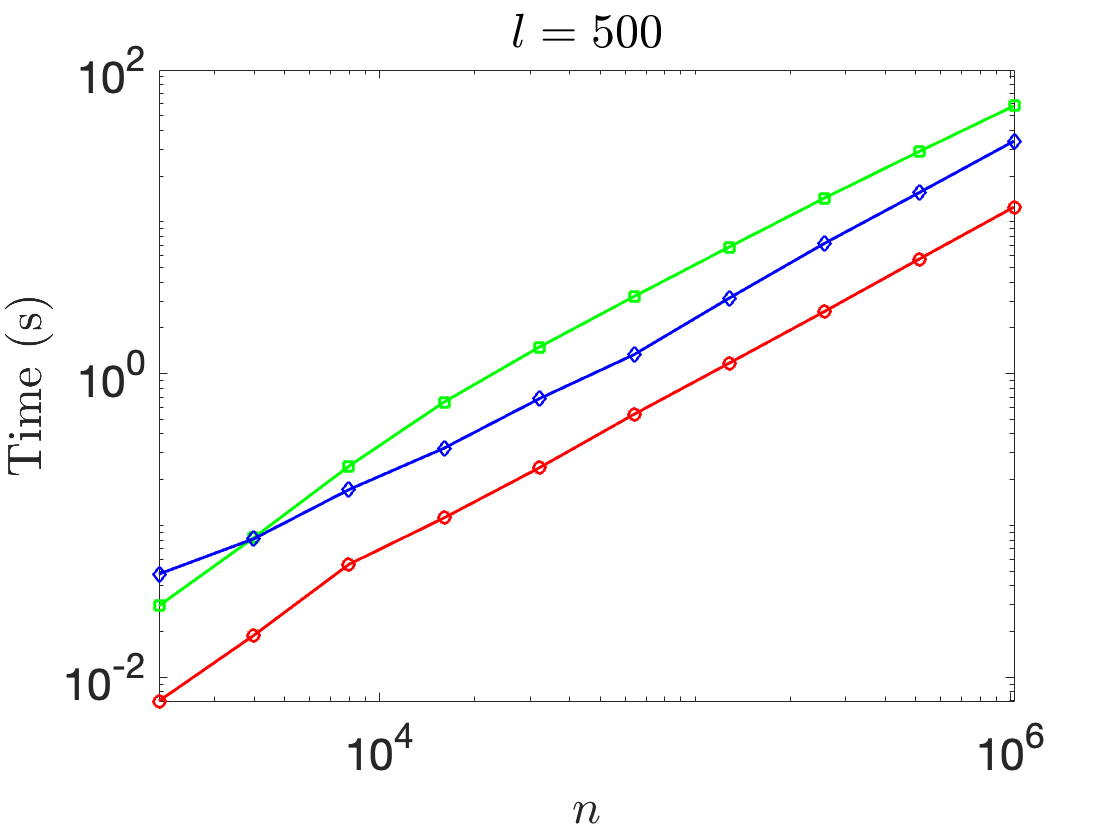}
    \end{subfigure}
    \caption{Run time of different pivoting schemes, scaled with respect to the problem size $n$, with different embedding dimension $l$.}
    \label{fig:pivot-ps-time}
\end{figure}
In Figure \ref{fig:pivot-ps-time}, we observe a considerable run time advantage of the \plu over the \pqr and DEIM, especially when $l$ is large. (Additionally, we see that DEIM slightly outperforms CPQR, which is perhaps surprising, given the substantially larger number of flops required by DEIM.)

\subsection{Randomized skeleton selection algorithms: accuracy and efficiency}\label{subsec:exp_cur}
As we move from measuring speed to measuring the precision of revealing the numerical rank of a matrix, 
the choice of test matrix becomes important. 
We consider four different classes of test matrices, including some synthetic random matrices with 
different spectral patterns, as well as some empirical datasets, as summarized below:
\begin{enumerate}
    \item \texttt{large}: a full-rank $4,282 \times 8,617$ sparse matrix with $20,635$ nonzero entries from the SuiteSparse matrix collection, generated by a linear programming problem sequence \cite{large}. 
    \item \texttt{YaleFace64x64}: a full-rank $165 \times 4096$ dense matrix, consisting of $165$ face images each of size $64 \times 64$. The flatten image vectors are centered and normalized such that the average image vector is zero, and the entries are bounded within $[-1,1]$.
    \item \texttt{MNIST} training set consists of $60,000$ images of hand-written digits from $0$ to $9$. Each image is of size $28 \times 28$. The images are flatten and normalized to form a full-rank matrix of size $N \times d$ where $N$ is the number of images and $d = 784$ is the size of the flatten images, with entries bounded in $[0,1]$. The nonzero entries take approximately $20\%$ of the matrix for both the training and the testing sets.
    \item Random \textit{sparse non-negative (SNN)} matrices are synthetic random sparse matrices used in \cite{voronin2017, sorensen2014} for testing skeleton selection algorithms. Given $s_1 \geq \dots \geq s_r > 0$, a random SNN matrix $\Ab$ of size $m \times n$ takes the form,
    \begin{equation}
        \label{eq:snn-def}
        \Ab = \text{SNN}\bpar{\cpar{s_i}_{i=1}^r;\ m,n}:= \sum_{i=1}^r s_i \xb_i \yb_i^T
    \end{equation}
    where $\xb_i \in \R^m$, $\yb_i \in \R^{n}$, $i \in [r]$ are random sparse vectors with non-negative entries.
    In the experiments, we use two random SNN matrices of distinct sizes:
    \begin{enumerate}[label=(\roman*)]
        \item \texttt{SNN1e3} is a $1000 \times 1000$ SNN matrix with $r = 1000$, $s_i = \frac{2}{i}$ for $i=1,\dots,100$, and $s_i = \frac{1}{i}$ for $i=101,\dots,1000$;
        \item \texttt{SNN1e6} is a $10^6 \times 10^6$ SNN matrix with $r = 400$, $s_i = \frac{2}{i}$ for $i=1,\dots,100$, and $s_i = \frac{1}{i}$ for $i=101,\dots,400$.
    \end{enumerate}
\end{enumerate}

Scaled with respect to the approximation ranks $k$, we compare the accuracy and efficiency of the following randomized CUR algorithms:
\begin{enumerate}
    \item Rand-\plu (and Rand-\plu-1piter):
    \Cref{algo:sketch_pivot_CUR_general} with $\Xb = \Gammab \Ab$ being a row sketch (or with one plain power iteration as in \Cref{eq:def_power_iter}), and pivoting with \plu;
    \item Rand-\pqr (and Rand-\pqr-1piter): \Cref{algo:sketch_pivot_CUR_general} with $\Xb = \Gammab \Ab$ being a row sketch (or with one power iteration as in \Cref{eq:def_power_iter}), and pivoting with \pqr (\cite{voronin2017});
    \item RSVD-DEIM: \Cref{algo:sketch_pivot_CUR_general} with $\Xb$ being an approximation of leading-$k$ right singular vectors (\Cref{eq:rsvd_procedures}), and pivoting with \plu (\cite{sorensen2014});
    \item RSVD-LS: Skeleton sampling based on approximated leverage scores (\cite{mahoney2009}) from a rank-$k$ SVD approximation (\Cref{eq:rsvd_procedures});
    \item SRCUR: Spectrum-revealing CUR decomposition proposed in \cite{chen2020}.
\end{enumerate}
The asymptotic complexities of the first three randomized pivoting based skeleton selection algorithms based on \Cref{algo:sketch_pivot_CUR_general} are summarized in \Cref{tab:complexity_rand_pivot}.
\begin{table}[!h]
    \centering
    \begin{tabular}{c|c|c}
    \hline
        Algorithm & Row space approximator construction (Line 1,2) & Pivoting (Line 3) \\
    \hline
        Rand-LUPP & $O(T_s(l,\Ab))$ & $O(n l^2)$ \\
        Rand-LUPP-1piter & $O(T_s(l,\Ab) + \nnz(\Ab) l)$ & $O(n l^2)$ \\
    \hline
        Rand-CPQR & $O(T_s(l,\Ab))$ & $O(n l^2)$ \\
        Rand-CPQR-1piter & $O(T_s(l,\Ab) + \nnz(\Ab) l)$ & $O(n l^2)$ \\
    \hline
        RSVD-DEIM & $O\rbr{T_s(l,\Ab) + (m+n)l^2 + \nnz(\Ab) l}$ & $O(nl^2)$ \\
    \hline
    \end{tabular}
    \caption{Asymptotic complexities of various randomized pivoting based skeleton selection algorithms based on \Cref{algo:sketch_pivot_CUR_general}.}
    \label{tab:complexity_rand_pivot}
\end{table}

For consistency, we use Gaussian embeddings for sketching throughout the experiments.
With the selected column and row skeletons, we leverage the stable construction in \Cref{eq:cur_stable} to form the corresponding CUR decompositions $\cur{\Ab}{I_s,J_s}$.
Although oversampling ($\ie$, $l > k$) is necessary for multiplicative error bounds with respect to the optimal rank-$k$ approximation error (\Cref{eq:rand_rangefinder_error_bound}, \Cref{thm:pivoting_on_rangeapprox_error_bound}), since oversampling can be interpreted as a shift of curves along the axis of the approximation rank, for the comparison purpose, we simply treat $l=k$, and compare the rank-$k$ approximation errors of the CUR decompositions against the optimal rank-$k$ approximation error $\norm{\Ab-\Ab_k}$.

\begin{figure}[!ht]
    \centering
    \begin{subfigure}{0.32\textwidth}
    \centering
    \includegraphics[width=\linewidth]{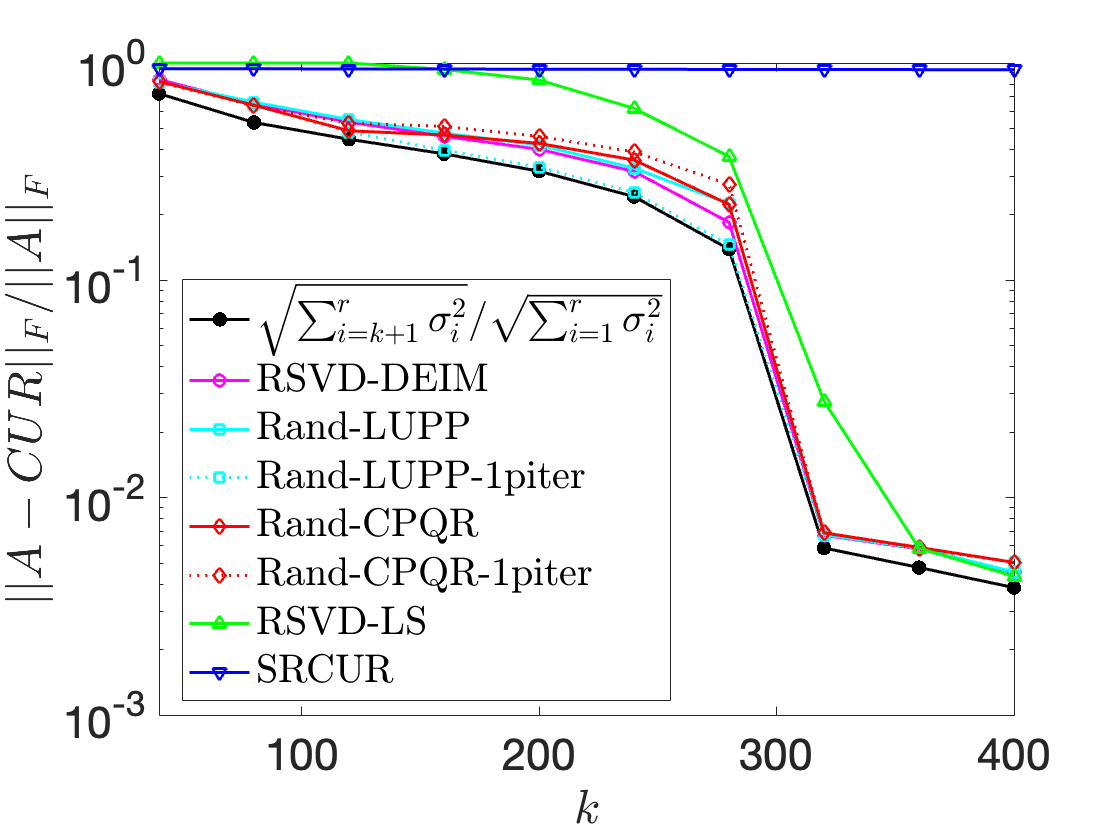}
    \caption{Frobenius norm error.}
    \label{fig:rand-errfro-rank_large}
    \end{subfigure}
    \begin{subfigure}{0.32\textwidth}
    \centering
    \includegraphics[width=\linewidth]{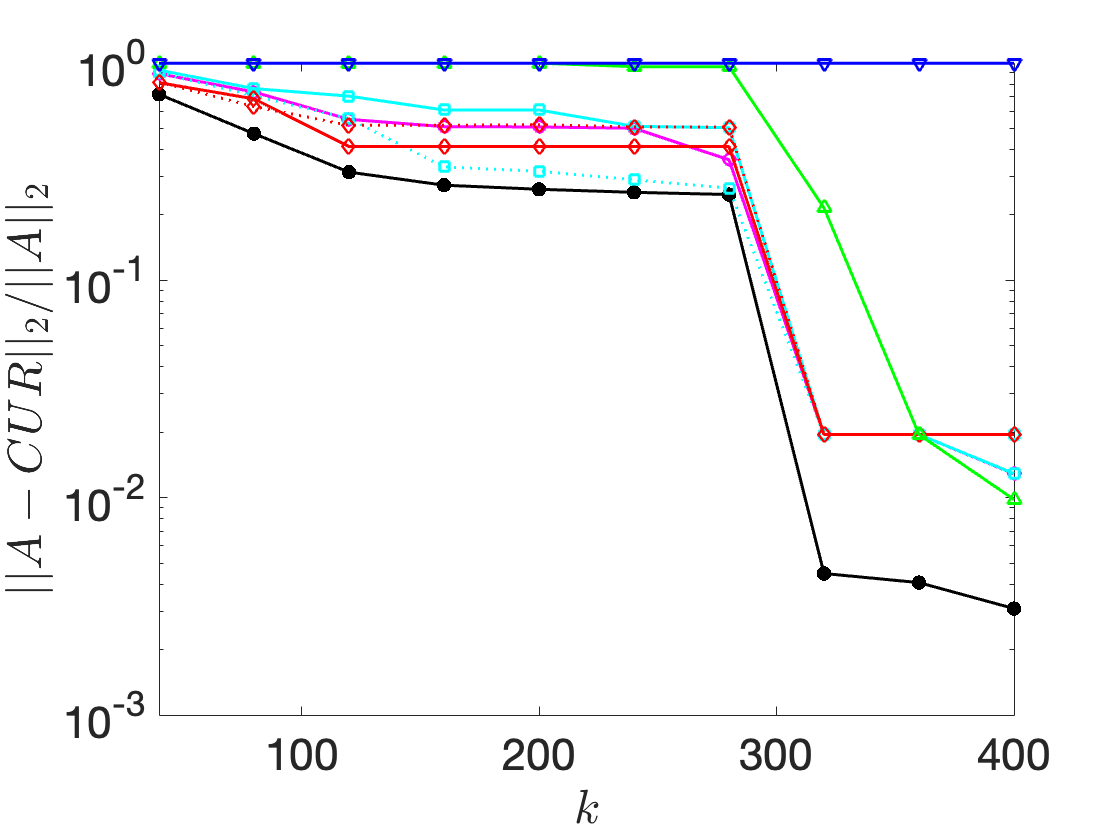}
    \caption{Spectral norm error.}
    \label{fig:rand-err2-rank_large}
    \end{subfigure}
    \begin{subfigure}{0.32\textwidth}
    \centering
    \includegraphics[width=\linewidth]{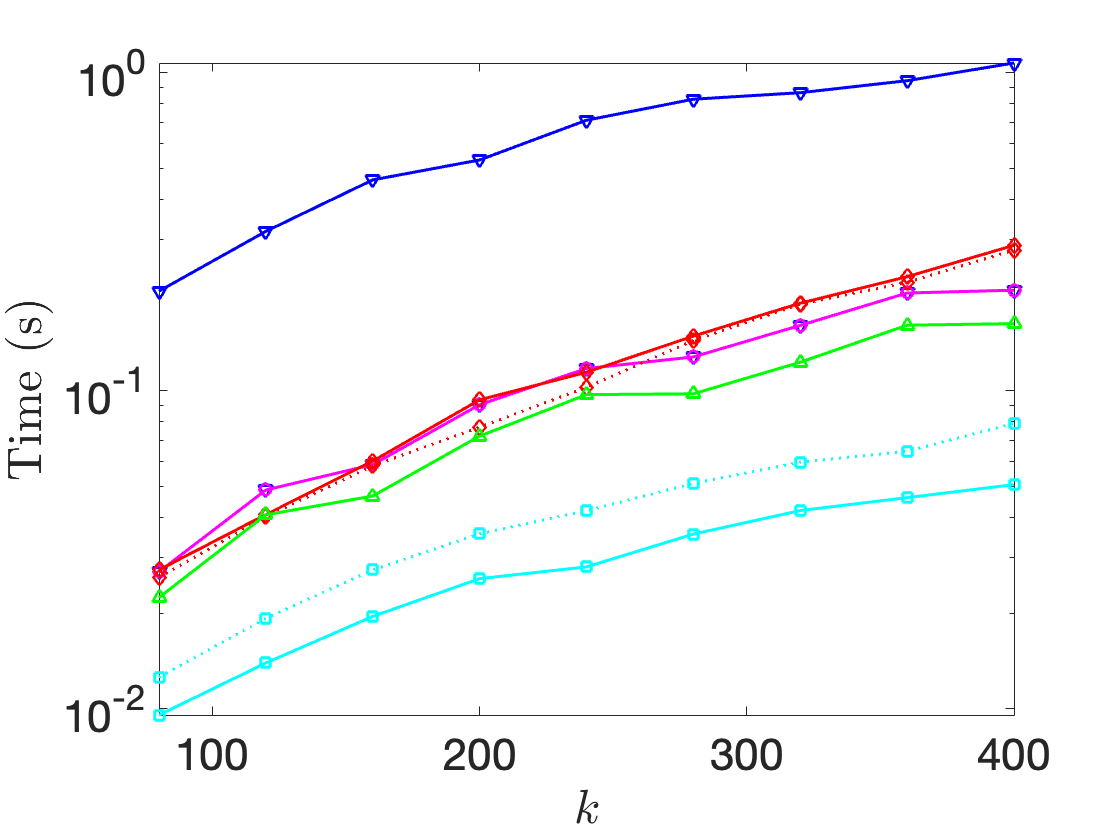}
    \caption{Runtime.}
    \label{fig:rand-time-rank_large}
    \end{subfigure}
    \caption{Relative error and run time of randomized skeleton selection on the \texttt{large} data set.}
    \label{fig:rand-err-rank_large}
\end{figure}

\begin{figure}[!ht]
    \centering
    \begin{subfigure}{0.32\textwidth}
    \centering
    \includegraphics[width=\linewidth]{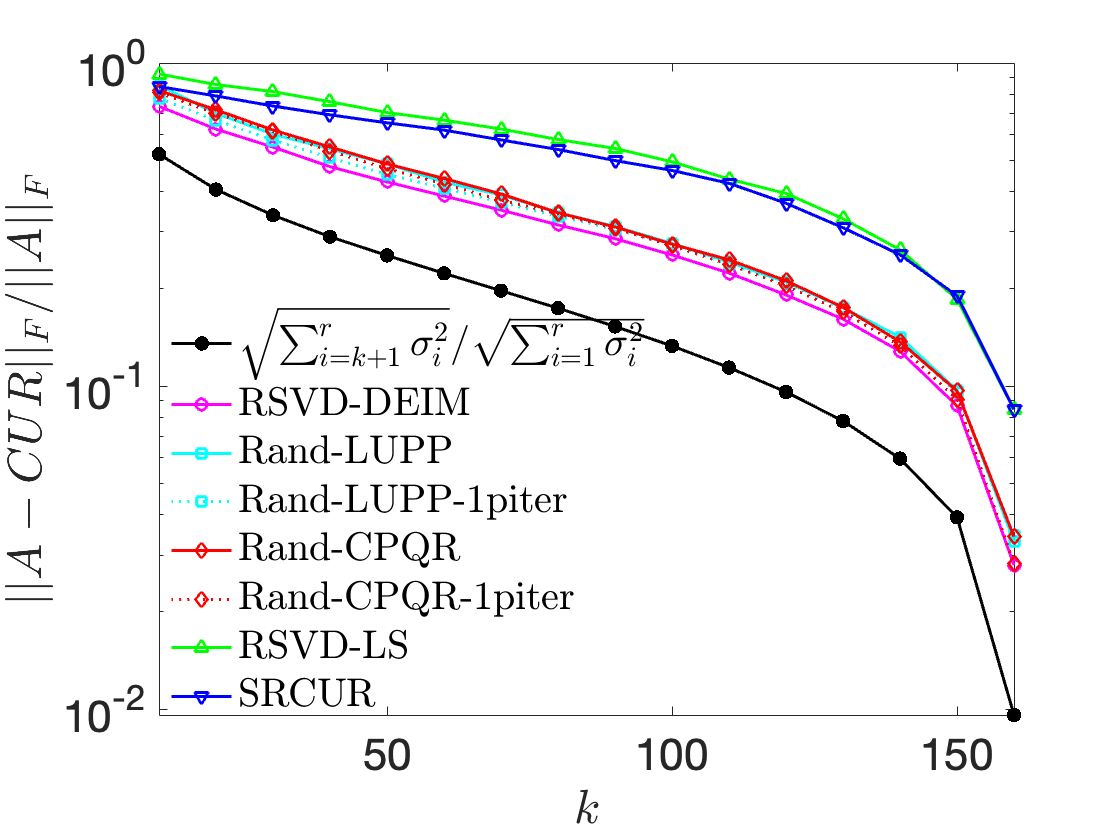}
    \caption{Frobenius norm error.}
    \label{fig:rand-errfro-rank_yaleface-64x64}
    \end{subfigure}
    \begin{subfigure}{0.32\textwidth}
    \centering
    \includegraphics[width=\linewidth]{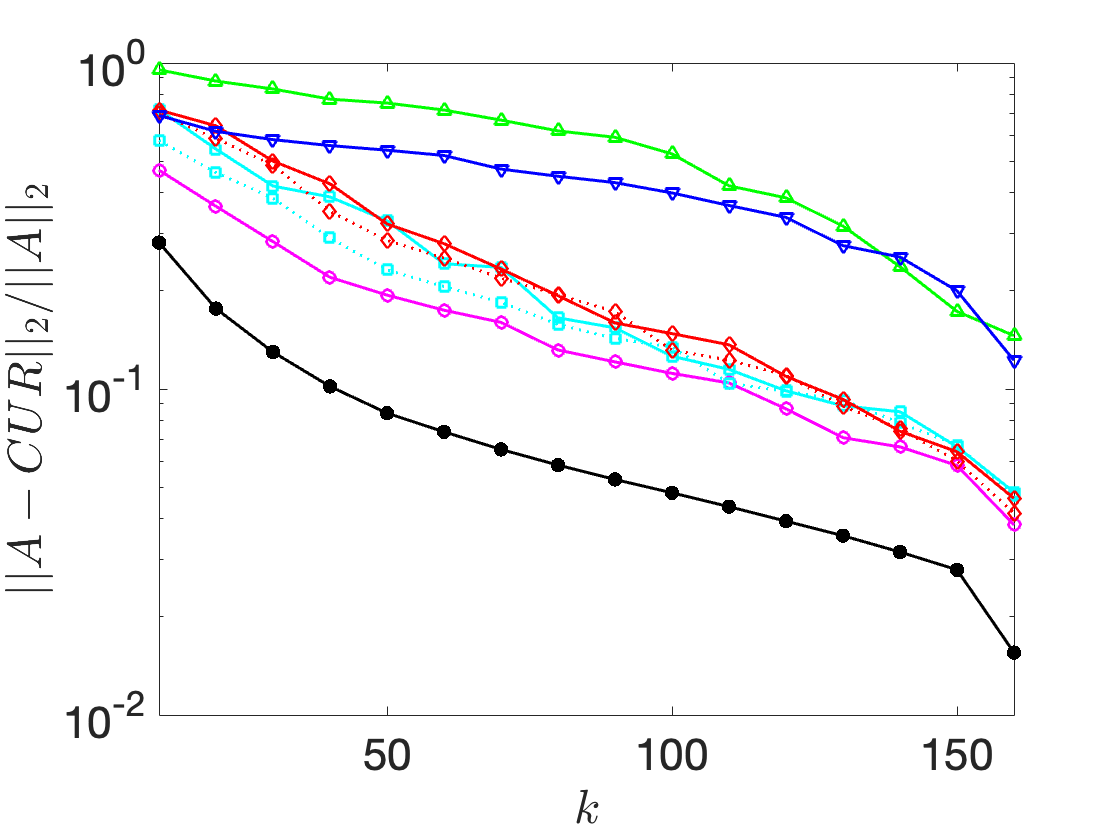}
    \caption{Spectral norm error.}
    \label{fig:rand-err2-rank_yaleface-64x64}
    \end{subfigure}
    \begin{subfigure}{0.32\textwidth}
    \centering
    \includegraphics[width=\linewidth]{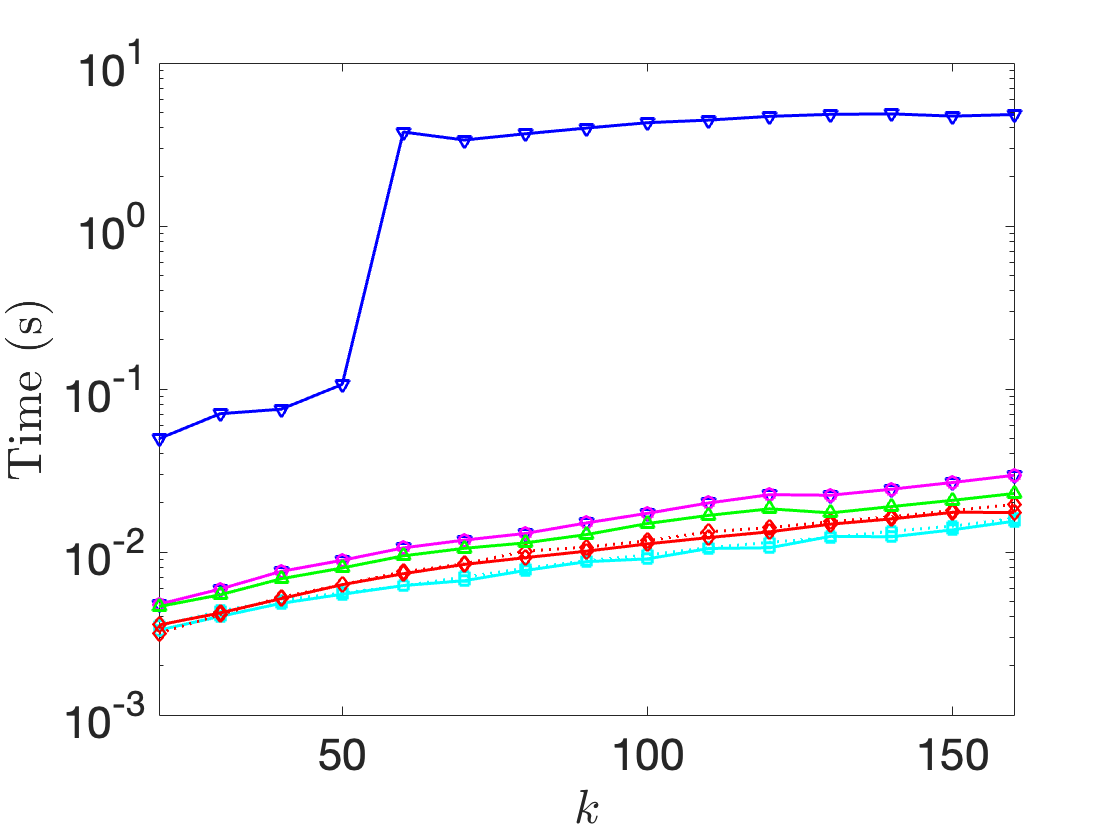}
    \caption{Runtime.}
    \label{fig:rand-time-rank_yaleface-64x64}
    \end{subfigure}
    \caption{Relative error and run time of randomized skeleton selection on the \texttt{YaleFace64x64} data set.}
    \label{fig:rand-err-rank_yaleface-64x64}
\end{figure}

\begin{figure}[!ht]
    \centering
    
    \begin{subfigure}{0.32\textwidth}
    \centering
    \includegraphics[width=\linewidth]{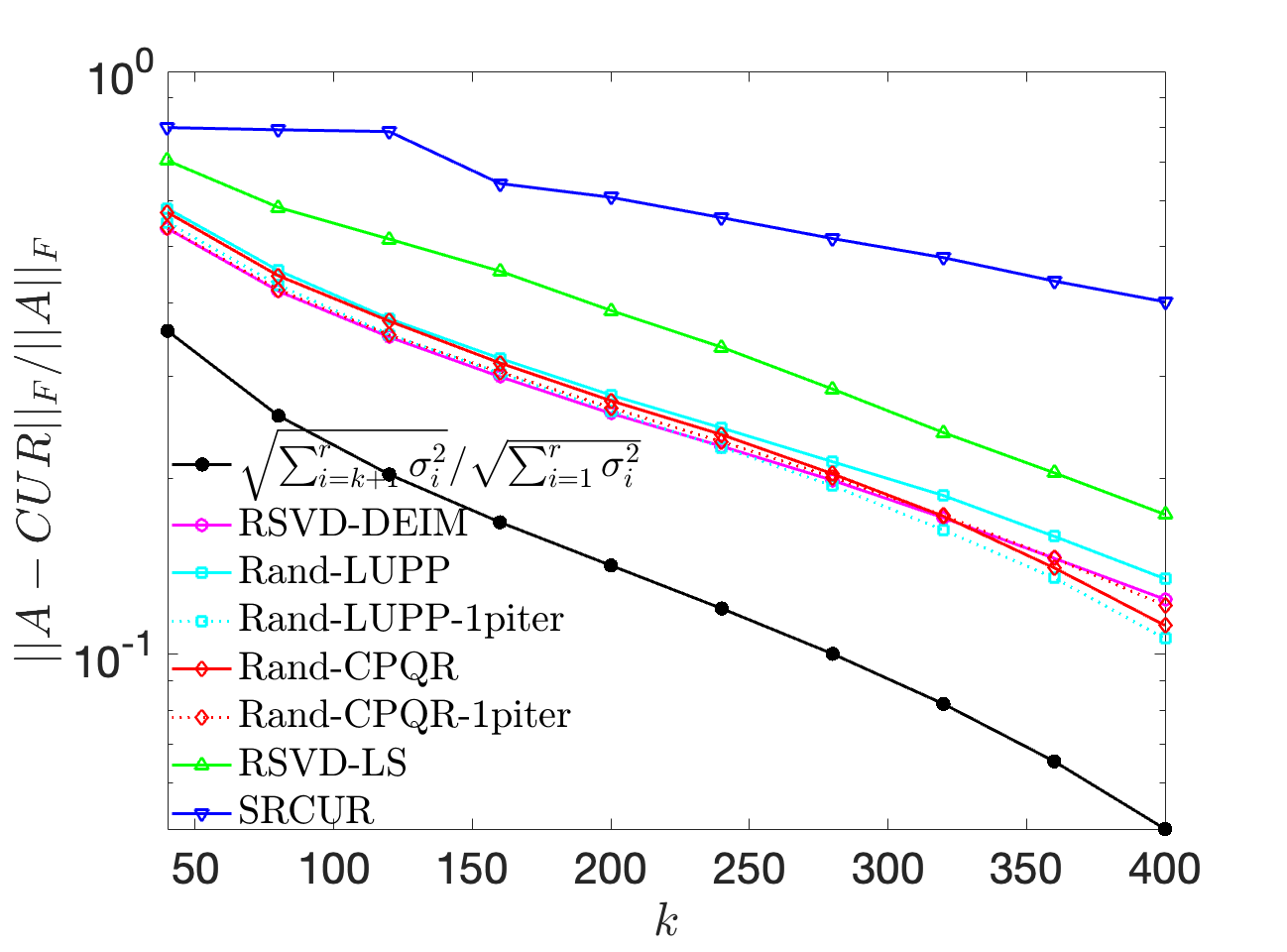}
    \caption{Frobenius norm error.}
    \label{fig:rand-errfro-rank_mnist-train}
    \end{subfigure}
    \begin{subfigure}{0.32\textwidth}
    \centering
    \includegraphics[width=\linewidth]{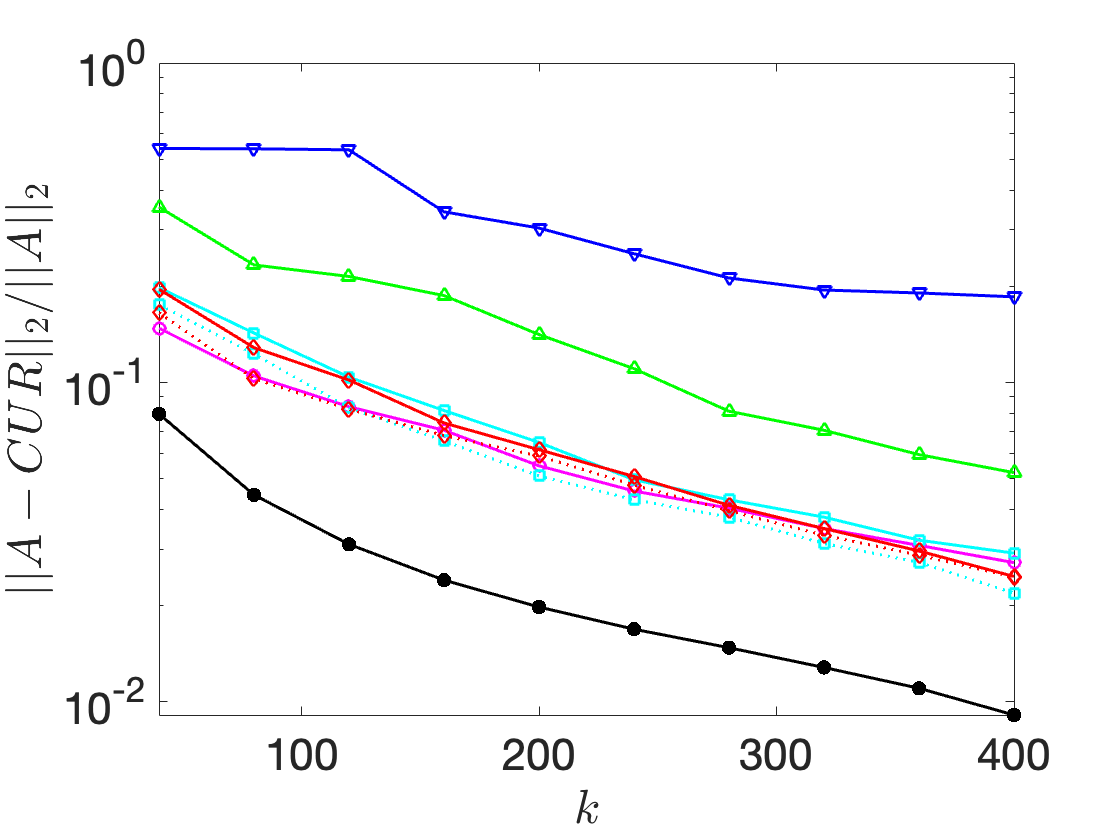}
    \caption{Spectral norm error.}
    \label{fig:rand-err2-rank_mnist-train}
    \end{subfigure}
    \begin{subfigure}{0.32\textwidth}
    \centering
    \includegraphics[width=\linewidth]{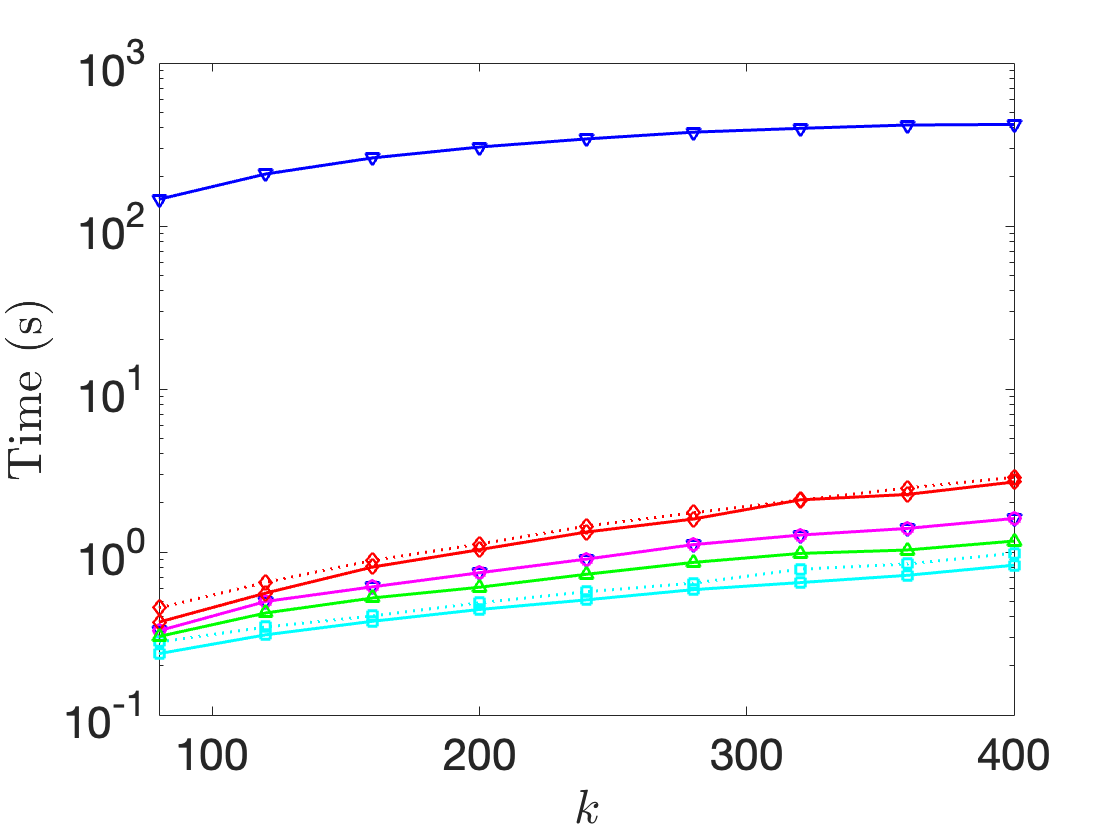}
    \caption{Runtime.}
    \label{fig:rand-time-rank_mnist-train}
    \end{subfigure}
    
    \caption{Relative error and run time of randomized skeleton selection on the training set of MNIST.}
    \label{fig:rand-err-rank_mnist-train}
\end{figure}

\begin{figure}[!ht]
    \centering
    
    \begin{subfigure}{0.32\textwidth}
    \centering
    \includegraphics[width=\linewidth]{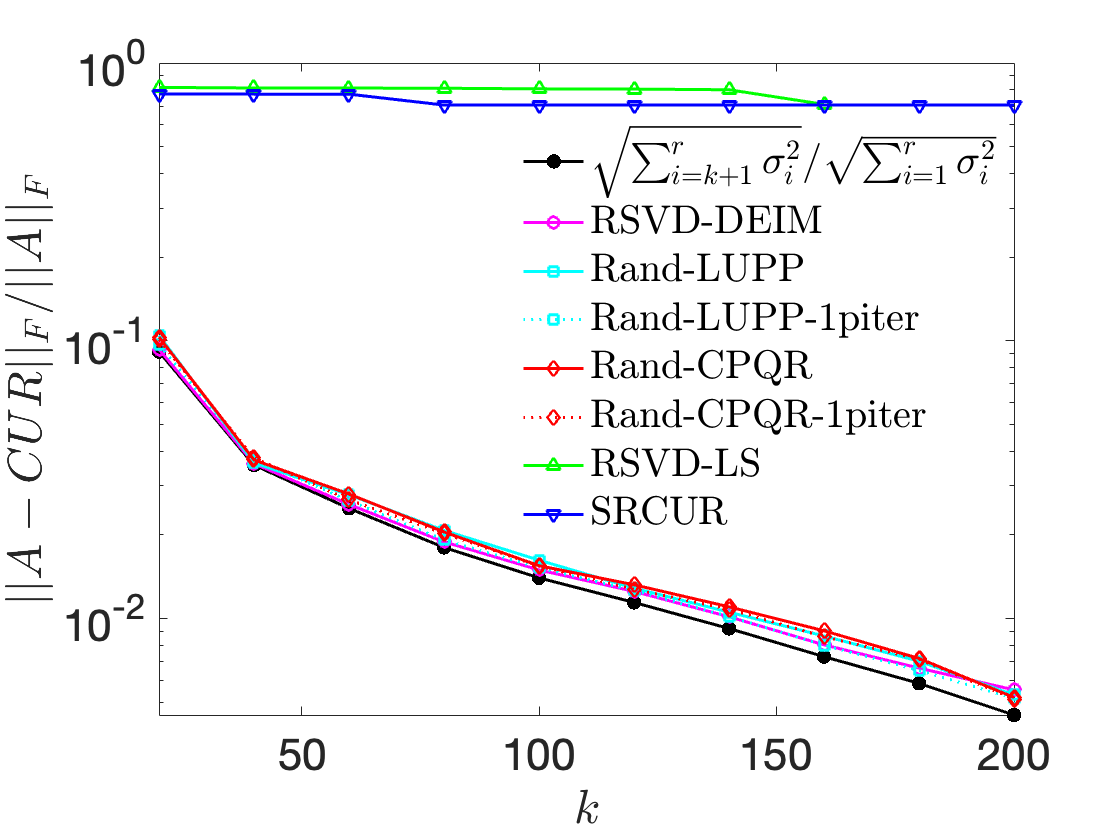}
    \caption{Frobenius norm error.}
    \label{fig:rand-errfro-rank_snn-1e3-1e3_a2b1_k100_r1e3_s1e-3}
    \end{subfigure}
    \begin{subfigure}{0.32\textwidth}
    \centering
    \includegraphics[width=\linewidth]{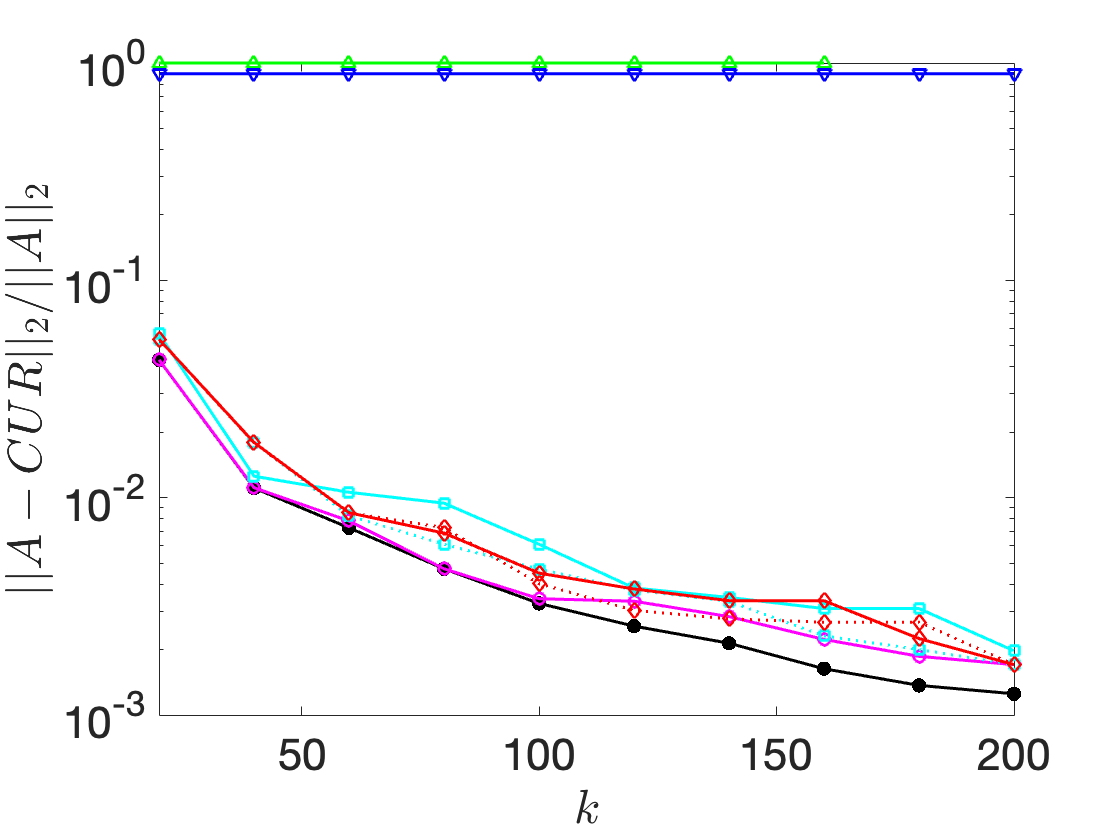}
    \caption{Spectral norm error.}
    \label{fig:rand-err2-rank_snn-1e3-1e3_a2b1_k100_r1e3_s1e-3}
    \end{subfigure}
    \begin{subfigure}{0.32\textwidth}
    \centering
    \includegraphics[width=\linewidth]{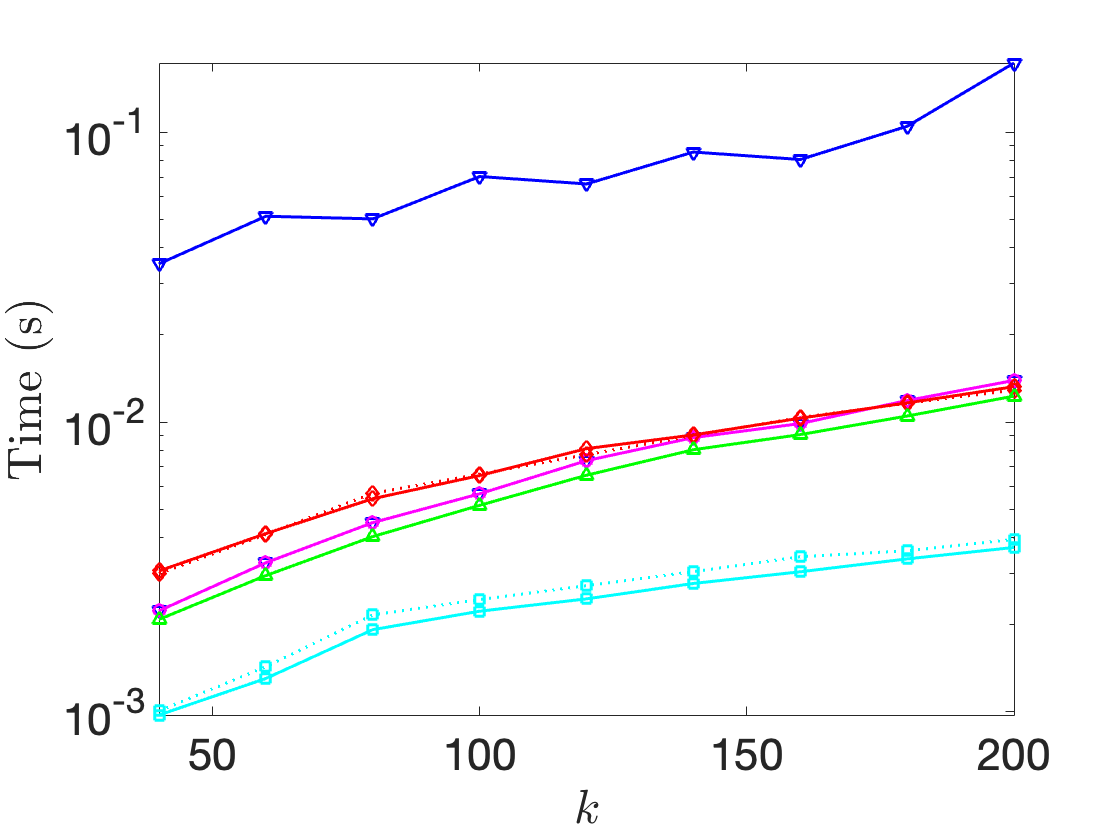}
    \caption{Runtime.}
    \label{fig:rand-time-rank_snn-1e3-1e3_a2b1_k100_r1e3_s1e-3}
    \end{subfigure}
    
    \caption{Relative error and run time of randomized skeleton selection on a $1000 \times 1000$ sparse non-negative random matrix, \texttt{SNN1e3}.}
    \label{fig:rand-err-rank_snn-1e3-1e3_a2b1_k100_r1e3_s1e-3}
\end{figure}

\begin{figure}[!ht]
    \centering
    
    \begin{subfigure}{0.32\textwidth}
    \centering
    \includegraphics[width=\linewidth]{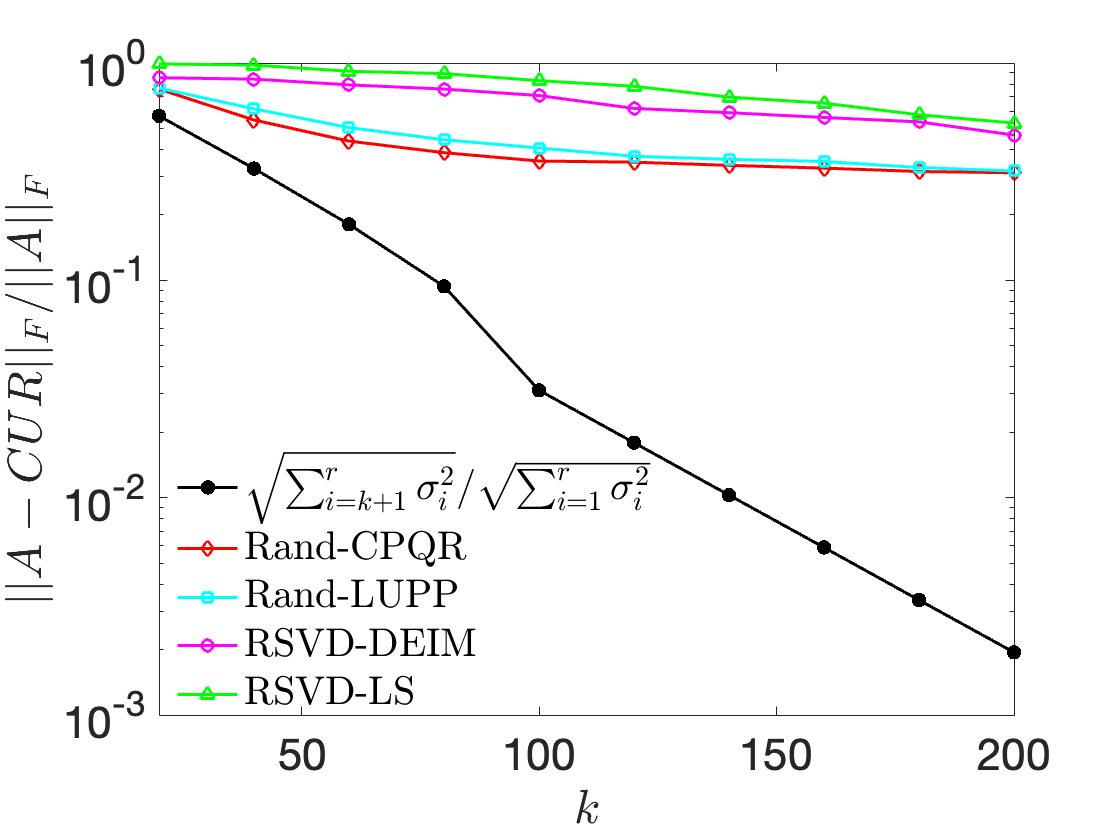}
    \caption{Frobenius norm error.}
    \label{fig:and-errref-rank_snn-1e6-1e6-a2b1-k100-r400-s2or}
    \end{subfigure}
    \begin{subfigure}{0.32\textwidth}
    \centering
    \includegraphics[width=\linewidth]{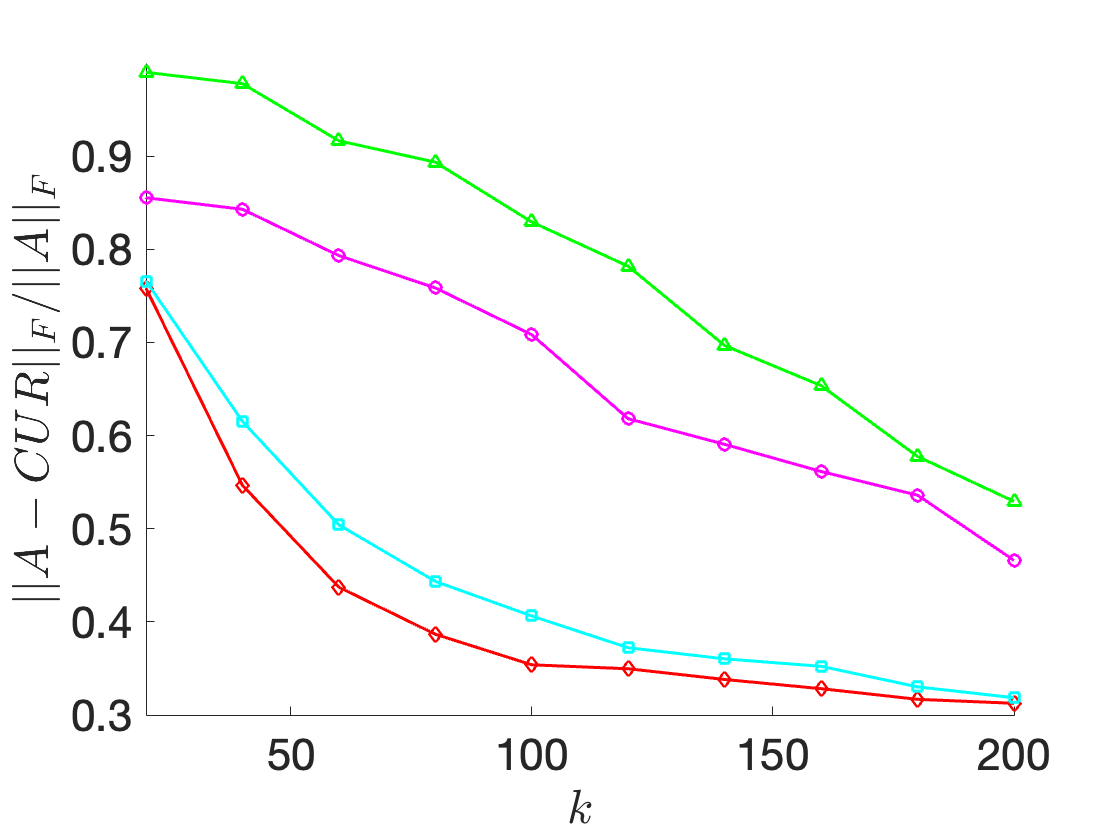}
    \caption{Frobenius norm error zoomed.}
    \label{fig:rand-errfro-rank_snn-1e6-1e6-a2b1-k100-r400-s2or}
    \end{subfigure}
    \begin{subfigure}{0.32\textwidth}
    \centering
    \includegraphics[width=\linewidth]{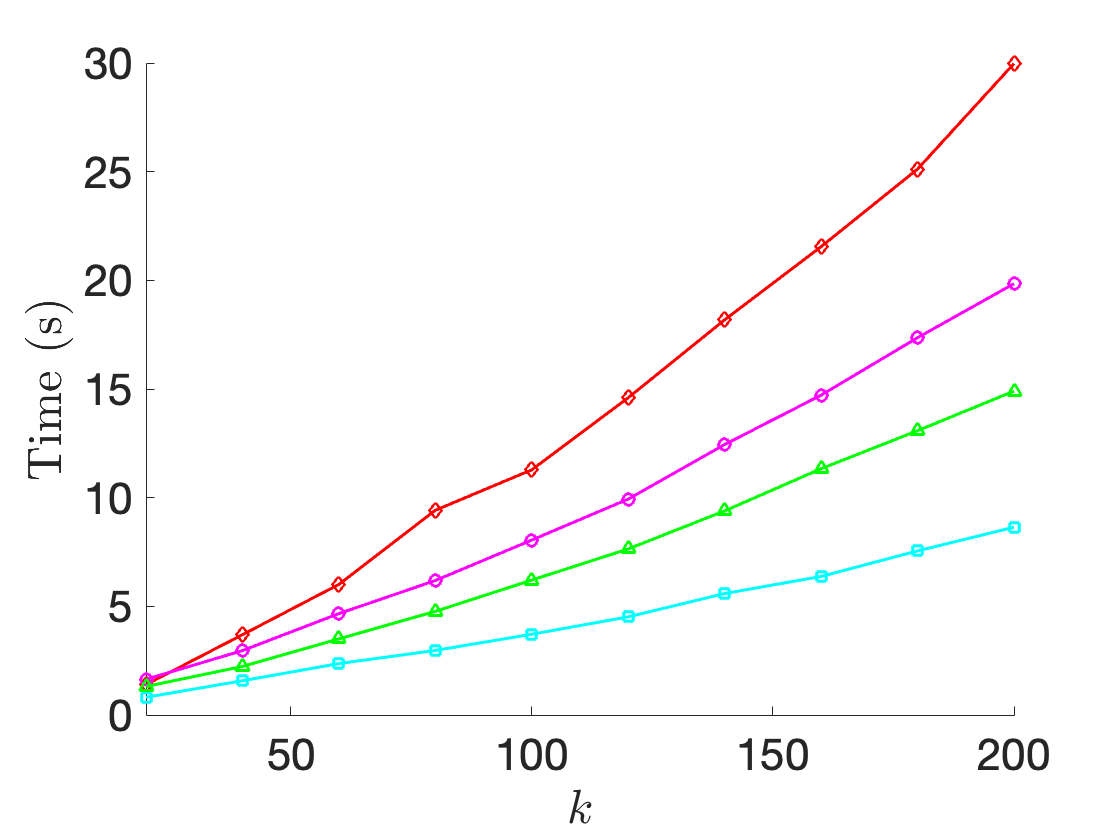}
    \caption{Runtime.}
    \label{fig:rand-time-rank_snn-1e6-1e6-a2b1-k100-r400-s2or}
    \end{subfigure}
    
    \caption{Relative error and run time of randomized skeleton selection on a $10^6 \times 10^6$ sparse non-negative random matrix, \texttt{SNN1e6}.}
    \label{fig:rand-err-rank_snn-1e6-1e6-a2b1-k100-r400-s2or}
\end{figure}

From Figure \ref{fig:rand-err-rank_large}-\ref{fig:rand-err-rank_snn-1e3-1e3_a2b1_k100_r1e3_s1e-3}, we observe that the randomized pivoting based skeleton selection algorithms that fall in \Cref{algo:sketch_pivot_CUR_general} ($\ie$, Rand-\plu, Rand-\pqr, and RSVD-DEIM) share the similar approximation accuracy, which is considerably lower than the RSVD-LS and SRCUR. From the efficiency perspective, Rand-\plu provides the most competitive run time among all the algorithms, especially when $\Ab$ is sparse. Meanwhile, we observe that, for both Rand-\pqr and Rand-\plu, constructing the sketches with one plain power iteration ($\ie$, with \Cref{eq:def_power_iter}) can observably improve the accuracy, without sacrificing the efficiency significantly ($\eg$, in comparison to the randomized DEIM which involves one power iteration with orthogonalization as in \Cref{eq:ortho_power_iter}).
In Figure \ref{fig:rand-err-rank_snn-1e6-1e6-a2b1-k100-r400-s2or}, the similar performance is also observed on a synthetic large-scale problem, \texttt{SNN1e6}, where the matrix is only accessible as a fast matrix-vector multiplication (matvec) oracle such that each matvec takes $o(mn)$ (i.e., $O((m+n)r)$ in our construction) operations.


\subsection*{Acknowledgments:}
The work reported was supported by the Office of Naval Research (N00014-18-1-2354),
by the National Science Foundation (DMS-1952735 and DMS-2012606),
and by the Department of Energy ASCR (DE-SC0022251).
The authors wish to thank Chao Chen, Ke Chen, Yuji Nakatsukasa, and Rachel Ward for valuable discussions.

\bibliographystyle{acm}
\setcitestyle{number}
\bibliography{reference}

\end{document}